\documentclass[a4paper, 12pt, reqno]{amsart}

\usepackage[normalem]{ulem} 
\usepackage{lmodern}
\usepackage[english]{babel}
\usepackage[latin1]{inputenc}
\usepackage{microtype}
\usepackage{amsfonts,amssymb,amsmath,amsthm,mathrsfs}
\usepackage{mathtools}
\usepackage{color, graphicx}
\usepackage[all]{xy}
\usepackage[bookmarks=false,colorlinks,urlcolor=cyan,citecolor=red,linkcolor=blue]{hyperref}  
\usepackage{version}
\usepackage[inline]{enumitem}
\usepackage{cleveref}
\usepackage{a4wide}
\usepackage{upgreek}
\usepackage{etoolbox}

\makeatletter
\patchcmd{\@setaddresses}{\indent}{\noindent}{}{}
\patchcmd{\@setaddresses}{\indent}{\noindent}{}{}
\patchcmd{\@setaddresses}{\indent}{\noindent}{}{}
\patchcmd{\@setaddresses}{\indent}{\noindent}{}{}
\makeatother

\makeatletter
\DeclareMathSizes{12}{12}{6}{5}
\makeatother

\makeatletter
\def\namedlabel#1#2{\begingroup
    #2%
    \def\@currentlabel{#2}%
    \phantomsection\label{#1}\endgroup
}
\makeatother

\allowdisplaybreaks

\theoremstyle{plain}

\newtheorem{theorem}{Theorem}[section]
\newtheorem{lemma}[theorem]{Lemma}
\newtheorem{proposition}[theorem]{Proposition}
\newtheorem{corollary}[theorem]{Corollary}
\newtheorem*{theorem*}{Theorem}

\theoremstyle{definition}
\newtheorem{definition}[theorem]{Definition}
\newtheorem{example}[theorem]{Example}

\theoremstyle{remark}
\newtheorem{remark}[theorem]{Remark}

\newcommand{\N}{\mathbb{N}}
\newcommand{\Z}{\mathbb{Z}}

\newcommand{\C}{\mathbb{C}}
\newcommand{\K}{\Bbbk}

\newcommand{\cl}[1]{\overline{#1}} 

\newcommand{\op}[1]{#1^{o}} 
\newcommand{\env}[1]{{#1}^{e}} 
\newcommand{\mf}[1]{\mathfrak{#1}} 
\newcommand{\mb}[1]{\boldsymbol{#1}} 

\newcommand{\tensor}[1]{\otimes_{{#1}}} 

\newcommand{\injlimit}[2]{\varinjlim_{#1}\left({#2}\right)} 
\renewcommand{\ker}{\mathsf{ker}} 

\newcommand{\id}{\mathsf{Id}} 

\newcommand{\Hom}[6]{{_{#1}^{#2}\mathsf{Hom}_{#3}^{#4}}\left({#5},{#6}\right)} 
\newcommand{\End}[2]{\mathsf{End}_{#1}(#2)} 
\newcommand{\Derk}[1]{\mathsf{Der}_{\K}\left({#1}\right)} 

\newcommand{\bialgk}{\mathsf{Bialg}_{\K}} 
\newcommand{\liek}{\mathsf{Lie}_{\K}} 
\newcommand{\vectk}{\mathsf{Vect}_{\K}} 
\newcommand{\Bim}[2]{{}_{#1}\M{}_{#2}} 
\newcommand{\FBim}[2]{{}^{\phantom{#1}}_{#1}\mathsf{FMod}_{#2}^{\phantom{#2}}} 
\newcommand{\GBim}[2]{{}^{\phantom{#1}}_{#1}\mathsf{GMod}_{#2}^{\phantom{#2}}} 
\newcommand{\Rmod}[1]{\M_{#1}} 
\newcommand{\Lmod}[1]{{}_{#1}\M} 
\newcommand{\Bimod}[1]{\Bim{#1}{#1}} 
\newcommand{\ALie}[1]{\mathsf{AnchLie}_{#1}} 
\newcommand{\Bialgd}{\mathsf{Bialgd}}

\newcommand{\prim}{\mathsf{Prim}}

\newcommand{\cB}{{\mathcal B}}
\newcommand{\cC}{{\mathcal C}}
\newcommand{\cD}{{\mathcal D}}

\newcommand{\cH}{{\mathcal H}}
\newcommand{\cI}{{\mathcal I}}

\newcommand{\cL}{{\mathcal L}}
\newcommand{\cM}{{\mathcal M}}

\newcommand{\cR}{{\mathcal R}}

\newcommand{\bB}{{\mathbb B}}

\newcommand{\bP}{{\mathbb P}}

\newcommand{\M}{\mathsf{Mod}} 
\newcommand{\calpha}{\mf{a}}
\newcommand{\clambda}{\mf{l}}
\newcommand{\crho}{\mf{r}}

\newcommand{\gr}{\mathsf{gr}} 
\newcommand{\ie}{i.e.~}

\newcommand{\ot}{\otimes}

\newbox\pullbackbox
\setbox\pullbackbox=\hbox{\xy 0;<1mm,0mm>: \POS(2,0)\ar@{-} (0,2) \ar@{-} (4,2)
\endxy}

\newbox\pushoutbox
\setbox\pushoutbox=\hbox{\xy 0;<1mm,0mm>: \POS(2,2)\ar@{-} (0,0) \ar@{-} (4,0)
\endxy}

\newcommand{\etaM}[1]{{}_{\eta} {#1}}

\newcommand{\sMto}[1]{{}_{s} {#1}{}_{\op{t}}}

\newcommand{\tak}[1]{\times_{{#1}}}

\definecolor{bostonuniversityred}{rgb}{0.8, 0.0, 0.0}

\title[Anchored Lie algebras and C-M's bialgebroids]{On anchored Lie algebras and the Connes-Moscovici's bialgebroid construction}

\author{Paolo Saracco}
\address{D\'epartement de Math\'ematique, Universit\'e Libre de Bruxelles, Boulevard du Triomphe, B-1050 Brussels, Belgium.}
\urladdr{sites.google.com/view/paolo-saracco}
\urladdr{homepages.ulb.ac.be/~psaracco}
\email{paolo.saracco@ulb.ac.be}

\thanks{Paolo Saracco is a collaborateur scientifique of the Fonds de la Recherche Scientifique - FNRS and a member of the ``National Group for Algebraic and Geometric Structures and their Applications'' (GNSAGA-INdAM)}
\keywords{Connes-Moscovici's bialgebroid; anchored Lie algebra; universal enveloping algebra; primitives; universal property}
\subjclass[2010]{Primary: 16B50; 16S10; 16S30; 16T15; 16W25; 18A40. Secondary: 16W70; 17A30; 17B66}

\begin{document}

\begin{abstract}
We show how the Connes-Moscovici's bialgebroid construction naturally provides universal objects for Lie algebras acting on non-commutative algebras.
\end{abstract}

\maketitle

\tableofcontents


\section*{Introduction}

Given a Hopf algebra $H$ (possibly with bijective antipode $S$) and a left $H$-module algebra $A$, one can turn the vector space $A \otimes H \otimes A$ into a left bialgebroid $\cH \coloneqq A \odot  H \odot  A$ over $A$ in a natural way. This procedure has been introduced independently, and under different forms, by Connes and Moscovici \cite{ConnesMoscovici} in their study of the index theory of transversely elliptic operators, 
and by Kadison \cite{Kad} in connection with his work on (pseudo-)Galois extensions. Later, Panaite and Van Oystaeyen proved in \cite{PanaiteOystaeyen} that the two constructions were in fact equivalent (isomorphic as $A$-bialgebroids) and that, as algebras, they were particular instances of the L-R-smash product introduced in \cite{PanaiteOystaeyen-LR}. Nevertheless, by following \cite{Bohm-handbook}, we will refer to the bialgebroid $A \odot  H \odot  A$ as the \emph{Connes-Moscovici's bialgebroid}.

Following the foregoing and, in particular, in view of the results in \cite{PanaiteOystaeyen}, two observations were made, that triggered the present investigation: \begin{enumerate*}[label=(\roman*)\hspace{-0.17cm},ref=(\roman*)]\item\label{item:intro1} that whenever a Lie algebra $L$ acts by derivations on an associative algebra $A$ (for the sake of simplicity, let us call it an \emph{$A$-anchored Lie algebra}), then $A$ becomes naturally an $U_\K(L)$-module algebra and \item\label{item:intro2} that the associated Connes-Moscovici's bialgebroid construction satisfies a universal property (both as $\env{A}$-ring and as $A$-bialgebroid, see \cite[Proposition 3.1 and Theorem 3.2]{PanaiteOystaeyen}) which suggests the possibility that $A \odot  U_\K(L) \odot  A$ plays for an $A$-anchored Lie algebra $L$ the same role played by the universal enveloping algebra for a Lie algebra.\end{enumerate*} 

Among anchored Lie algebras we find the well-known Lie-Rinehart algebras, which are in particular Lie algebras acting on commutative algebras.
As it can be inferred from the substantial literature on the topic, Lie-Rinehart algebras are a deeply investigated area, in particular for its connections with differential geometry (the global sections of a Lie algebroid $\cL \to \cM$ form a Lie-Rinehart algebra over $\cC^{\infty}(\cM)$). Rinehart himself gave an explicit construction of the universal enveloping algebra $U(R,L)$ of a Lie-Rinehart algebra in \cite{Rinehart} and proved a Poincar\'e-Birkhoff-Witt theorem for the latter. Other equivalent constructions are provided in \cite[\S3.2]{LaiachiPaolo-complete}, \cite[page 64]{Huebschmann-Poisson}, \cite[\S18]{Sweedler-groups}. The universal property of $U(R,L)$ as an algebra is spelled out in \cite[page 64]{Huebschmann-Poisson} and \cite[page 174]{Malliavin} (where it is attributed to Feld'man). Its universal property as an $A$-bialgebroid is codified in the Cartier-Milnor-Moore Theorem for $U(R,L)$ proved in \cite[\S3]{MoerdijkLie}, where Moerdijk and Mr\v{c}un show that the construction of the universal enveloping algebra provides a left adjoint to the functor associating any cocommutative bialgebroid with its Lie-Rinehart algebra of primitive elements and they find natural conditions under which this adjunction becomes an equivalence (as it has been done in \cite{MilnorMoore} for cocommutative bialgebras and Lie algebras). Further algebraic and categorical properties and applications are investigated in \cite{ArdiLaiachiPaolo,LaiachiPaolo-diff,Huebschmann-quantization,Huebschmann-LR,Huebschmann-Poisson}.

However, there are many important examples of Lie algebras acting by derivations on associative algebras which are not necessarily commutative (actually, any Lie algebra acts by derivations on its universal enveloping algebra and any associative algebra acts by inner derivations on itself). Furthermore, while the space of primitive elements of a bialgebra is always a Lie algebra and a primitively generated bialgebra is always cocommutative, the space of primitive elements of a bialgebroid is not, in general, a Lie-Rinehart algebra and not every primitively generated bialgebroid is necessarily cocommutative. A third observation that stood up for the present investigation is that, instead, the space of primitive elements of a bialgebroid is always a Lie algebra acting by derivations on the base algebra.

These facts, together with the two foregoing observations \ref{item:intro1} and \ref{item:intro2}, called for the study of Lie algebras $L$ acting on non-commutative algebras $A$ in their own right and, in particular, for the study of the associated Connes-Moscovici's bialgebroid $A \odot  U_\K(L) \odot  A$, as it has been done for Lie-Rinehart algebras and their universal enveloping algebras.

In the present paper, we are mainly concerned with two universal properties of $A \odot  U_\K(L) \odot  A$, as an $\env{A}$-ring and as an $A$-bialgebroid, which reflect the two well-known universal properties of universal enveloping algebras reported above. The first one (Theorem \ref{thm:UEA}) exhibits $A \odot  U_\K(L) \odot  A$ as the universal $\env{A}$-ring associated with the $A$-anchored Lie algebra $L$, similarly to what happens for $U(R,L)$ in \cite[page 64]{Huebschmann-Poisson}. Namely, for any $\env{A}$-ring $\phi_A:\env{A} \to \cR$ and any $\K$-Lie algebra morphism $\phi_L:L \to \cL(\cR)$ such that
\[\big[\phi_L(X),\phi_A(a \otimes \op{b})\big] = \phi_A\left(X \cdot (a \otimes \op{b})\right),\]
for all $a,b \in A$ and all $X \in L$, there exists a unique morphism of $\env{A}$-rings $\Phi:\cB_L \to \cR$ extending $\phi_L$. This naturally affects the study of the representations of $L$ (see Corollary \ref{cor:representations}). The second universal property (Proposition \ref{prop:counit}) exhibits $A \odot  U_\K(L) \odot  A$ as the universal $A$-bialgebroid associated with the $A$-anchored Lie algebra $L$, similarly to what happens for $U(R,L)$ in \cite[Theorem 3.1(i)]{MoerdijkLie}. Namely, for any $A$-bialgebroid $\cB$ and any morphism of $\K$-Lie algebras $\phi_L:L \to \cB$ which lands into the space $\prim(\cB)$ of primitive elements of $\cB$ and that is compatible with the anchors, there exists a unique morphism of $A$-bialgebroids $\Phi:\cB_L \to \cB$ that extends $\phi_L$.

Concretely, after a first section devoted to recall some definitions and some preliminary results, we introduce $A$-anchored Lie algebras in \S\ref{ssec:Aanch} and we prove that the Connes-Moscovici's bialgebroid associated to an $A$-anchored Lie algebra satisfies the stated universal property as $\env{A}$-ring in \S\ref{ssec:UEA} (Theorem \ref{thm:UEA}). In \S\ref{ssec:primitives}, we detail how taking the space of primitives of an $A$-bialgebroid induces a functor from the category of $A$-bialgebroids to the category of $A$-anchored Lie algebras and in \S\ref{ssec:adj} we show that the Connes-Moscovici's construction provides a natural left adjoint to this latter functor (Theorem \ref{th:main}) and, at the same time, we prove the second universal property of $A \odot U_\K(L) \odot A$ (Proposition \ref{prop:counit}). At this point, by finding inspiration from the Milnor-Moore and the Moerdijk-Mr\v{c}un theorems, we look for intrinsic conditions on a bialgebroid that allow us to recognize it as a $A \odot U_\K(L) \odot A$ for a certain $A$-anchored Lie algebra. Section \ref{sec:intrinsic} is devoted to find a first answer (Theorem \ref{thm:MM}). After studying in more detail the space of primitives of a Connes-Moscovici's bialgebroid in \S\ref{ssec:CMprimitives}, we tackle the question in the general framework in \S\ref{ssec:primgen} and in the particular case of bialgebroids over a commutative base in \S\ref{ssec:commutative}. Finally, we conclude with some final remarks about future lines of investigation in \S\ref{ssec:finalremarks}.

\subsection*{Notation}

All over the paper, we assume a certain familiarity of the reader with the language of monoidal categories and of (co)monoids therein (see, for example, \cite[VII]{MacLane}).

We work over a ground field $\K$ of characteristic $0$. All vector spaces are assumed to be over $\K$. The unadorned tensor product $\otimes$ stands for $\tensor{\K}$. All (co)algebras and bialgebras are intended to be  $\K$-(co)algebras and $\K$-bialgebras, that is to say, (co)algebras and bialgebras in the symmetric monoidal category of vector spaces $\left(\vectk,\otimes ,\K\right) $. Every (co)module has an underlying vector space structure. Identity morphisms $\id_V$ are often denoted simply by $V$.

In order to avoid confusion between indexes of elements and coproducts or coactions, we will adopt the following variant of the Heyneman-Sweedler's Sigma Notation. For $c$ in a coalgebra $C$, $m$ in a left $C$-comodule $M$ and $n$ in a right $C$-comodule $N$ we write
\begin{equation*}
\Delta_C \left( c\right) = \sum c_{\left( 1\right) }\otimes c_{\left( 2\right) }, \qquad \rho_M(m) = \sum m_{(-1)} \otimes m_{(0)} \qquad \text{and} \qquad \rho_N(n) = \sum n_{(0)} \otimes n_{(1)}.
\end{equation*}

Given an algebra $A$, we denote by $\op{A}$ its opposite algebra. We freely use the canonical isomorphism between the category of left $A$-module $\Lmod{A}$ and that of right $\op{A}$-modules $\Rmod{\op{A}}$. We also set $\env{A}\coloneqq A\tensor{}\op{A}$ and we identify the category of left $\env{A}$-modules $\Lmod{\env{A}}$ with that of $A$-bimodules $\Bimod{A}$. 
Recall that any morphism of algebras $\eta:\env{A} \to R$ leads to two commuting algebra maps $s: A \to R$, $a \mapsto a\tensor{}\op{1},$ and $t:\op{A}\to R$, $\op{a} \mapsto 1\tensor{} \op{a},$ (\ie $s(a)t(\op{b}) = t(\op{b})s(a)$ for all $a,b \in A$) and conversely.  
Given two $R$-bimodules $M$ and $N$, this gives rise to several $A$-module structures on $M$ and $N$ and it leads to several ways of considering the tensor product over $A$ between the underlying $A$-bimodules. In the present paper we focus on the $A$-bimodule structure induced by $\env{A}$ acting on the left via $\eta$ and we denote by $\sMto{M}$. We usually consider this bimodule structure when taking tensor products. If we want to stress the fact that $M$ is considered as a left $\env{A}$-module, we may also write $\etaM{M}$. Therefore, given two $R$-bimodules $M$ and $N$, we consider the tensor product $A$-bimodule 
\begin{equation}\label{eq:tensA}
M \tensor{A} N \coloneqq  \sMto{M} \tensor{A}  \sMto{N} = \frac{M \ot N}{\Big\langle t(\op{a})m \ot n - m \ot s(a)n ~\big\vert~ m\in M, n\in N, a\in A \Big\rangle}.
\end{equation}
Inside $M \tensor{A} N$, we will also consider the distinguished subspace
\begin{equation}\label{eq:takA}
M \tak{A} N \coloneqq \left\{ \left. \sum_i m_i \tensor{A} n_i \in M\tensor{A}N \ \right| \ \sum_i m_it(\op{a}) \tensor{A} n_i = \sum_i m_i \tensor{A} n_is(a)\right\},
\end{equation}
which is often called \emph{Takeuchi-Sweedler's $\times$-product}. It is an $A$-subbimodule with actions
\[
a\cdot \left(\sum_i m_i \tensor{A} n_i\right) = \sum_i s(a)m_i \tensor{A} n_i \qquad \text{and} \qquad \left(\sum_i m_i \tensor{A} n_i\right)\cdot a = \sum_i m_i \tensor{A} t(\op{a})n_i
\]
for all $\sum_i m_i \tensor{A} n_i\in M\tak{A}N$ and $a\in A$ (see \cite[Definition 2.1]{Sweedler-groups} and \cite[page 460]{Takeuchi}). In particular, the following relations hold for all $m \in M$, $n \in N$, $\sum_i m_i \tensor{A} n_i \in M\tak{A} N$ and for all $a\in A$:
\begin{equation}\label{eq:tak}
t(\op{a})m \tensor{A} n \stackrel{\eqref{eq:tensA}}{=} m \tensor{A} s(a)n \qquad \text{and} \qquad \sum_i m_it(\op{a}) \tensor{A} n_i \stackrel{\eqref{eq:takA}}{=} \sum_i m_i \tensor{A} n_is(a).
\end{equation}
For the sake of clarity, it will be useful to set
\begin{equation}\label{eq:triangles}
a\triangleright m \triangleleft b = s(a)t(\op{b})m = \eta(a \otimes \op{b})m \quad \text{and} \quad b\blacktriangleright m \blacktriangleleft a = ms(a)t(\op{b}) = m\eta(a \otimes \op{b})
\end{equation}
for $m \in {{}_\eta M{}_\eta}$ and $a,b \in A$.


\section{Preliminaries}

We begin by collecting some facts about bimodules, corings and bialgebroids that will be needed in the sequel. The aim is that of keeping the exposition self-contained. Many results and definitions we will present herein hold in a more general context and under less restrictive hypotheses, but we preferred to limit ourselves to the essentials.

Given a (preferably, non-commutative) $\K$-algebra $A$, the category of $A$-bimodules forms a non-strict monoidal category $\left(\Bimod{A},\tensor{A},A,\calpha,\clambda,\crho\right)$. 
Nevertheless, all over the paper we will behave as if the structural natural isomorphisms
\begin{gather*}
\calpha_{M,N,P}:\left(M\tensor{A}N\right)\tensor{A}P \to M\tensor{A}\left(N\tensor{A}P\right), \quad \left(m\tensor{A}n\right)\tensor{A}p \to m\tensor{A}\left(n\tensor{A}p\right), \\
\clambda_M: A \tensor{A} M \to M, \quad a\tensor{A}m \mapsto a\cdot m, \qquad \text{and} \qquad  \crho_M:M\tensor{A}A \to M, \quad m \tensor{A} a \mapsto m \cdot a,
\end{gather*}
were ``the identities'', that is, as if $\Bimod{A}$ was a strict monoidal category.


\subsection{Graded and filtered \texorpdfstring{$A$}{A}-bimodules}\label{ssec:filtgrad}

As far as we are concerned, we assume $A$ to be filtered over $\Z$ with filtration 
\[F_{n}(A) = 0 \quad \text{for all } n < 0 \qquad \text{and} \qquad F_{n}(A) = A \quad \text{for all } n \geq 0\]
and we assume it to be graded over $\Z$ with graduation
\[A_0 = A \qquad \text{and} \qquad A_n = 0 \quad \text{for all } n \neq 0.\]
By a \emph{graded $A$-bimodule} we mean an $A$-bimodule $M$ with a family of $A$-subbimodules $\left\{M_n\mid n \in \Z\right\}$ such that $M = \bigoplus_{n \in \Z}M_n$. By a \emph{filtered $A$-bimodule} we mean an $A$-bimodule $M$ with a chain of $A$-subbimodules $\left\{F_n(M)\mid n \in \Z\right\}$, that is $F_p(M) \subseteq F_q(M)$ if $p \leq q$. The filtration is said to be \emph{exhaustive} if $M = \bigcup_{n \in \Z} F_n(M)$. Given a filtered $A$-bimodule $M$, one can consider the \emph{associated graded} bimodule $\gr(M)$ defined by
\[\gr_n(M)\coloneqq \frac{F_n(M)}{F_{n-1}(M)} \qquad \text{and} \qquad \gr(M) \coloneqq \bigoplus_{n \in \Z} \gr_n(M).\]
In what follows we will be interested in \emph{positively} filtered and graded bimodules, that is, those for which the negative terms are $0$. 

Given two filtered bimodules $M,N$, we can perform their tensor product $M \tensor{A} N$ and this is still a filtered bimodule: the $k$-th term of the filtration on $M \tensor{A} N $ is the $A$-subbimodule generated by the elements $m \tensor{A} n$ such that $m \in F_s(M)$, $n \in F_t(N)$ and $s+t=k$. Analogously for two graded $A$-bimodules. With these tensor products, we have that the categories $\left(\FBim{A}{A},\tensor{A},A\right)$ and $\left(\GBim{A}{A},\tensor{A},A\right)$ of filtered and graded bimodules, respectively, are monoidal categories (it follows, for instance, from \cite[Chapter A, Proposition I.2.14 and Chapter D, Lemma VIII.1]{NastasescuOystaeyen}). Morphisms of filtered (respectively, graded) bimodules are $A$-bilinear maps that respect the filtration (respectively, graduation).

The result we are principally interested in is that the construction of the graded associated to a filtered bimodule is functorial (see, for example, \cite[chapter D, \S III]{NastasescuOystaeyen}). Moreover, the natural surjection
\begin{equation}\label{eq:laxmon}
\varphi_{M,N} \colon \gr(M) \tensor{A} \gr(N) \to \gr(M\tensor{A}N),
\end{equation}
uniquely determined by
\[ \left(m + F_{s-1}(M)\right)\tensor{A}\left(n + F_{t-1}(N)\right) \mapsto \left(m \tensor{A}n\right) + F_{s+t-1}(M \tensor{A}N) \]
for $m \in F_s(M)$ and $n \in F_t(N)$ (see \cite[page 318]{NastasescuOystaeyen}), and the isomorphism $\varphi_0:A \cong \gr(A)$ endow the functor
\[\gr(-):\FBim{A}{A} \to \GBim{A}{A}, \qquad  M \mapsto \gr(M),\]
with a structure of lax monoidal functor (see \cite[Definition 3.1]{Aguiar}). For further details about filtered and graded bimodules, we refer the reader to \cite{NastasescuOystaeyen}.


\subsection{\texorpdfstring{$A$}{A}-corings}\label{ssec:Acorings}

Recall that an $A$-coring is a monoid in the monoidal category of $A$-bimodules $\left(\Bimod{A},\tensor{A},A\right)$. More concretely, an $A$-coring is an $A$-bimodule $\cC$ endowed with a comultiplication $\Delta_\cC : \cC \to \cC \tensor{A} \cC$ and a counit $\varepsilon_\cC:\cC \to A$ such that
\begin{equation}\label{eq:coasscoun}
\left(\Delta_\cC \tensor{A} \cC\right) \circ \Delta_\cC = \left(\cC \tensor{A} \Delta_\cC\right) \circ \Delta_\cC \quad \text{and} \quad \left(\varepsilon_\cC\tensor{A} \cC\right)\circ \Delta_\cC = \cC = \left(\cC\tensor{A} \varepsilon_\cC\right)\circ \Delta_\cC.
\end{equation}
For the general theory of corings and their comodules, we refer to \cite{BrzezinskiWisbauer}.

Later on, we will be particularly interested in (exhaustively) filtered $A$-corings such that the associated graded components are projective as $A$-bimodules. These are $A$-corings $\cC$ endowed with an increasing filtration $\left\{F_n(\cC) \mid n \in \N \right\}$ as $A$-bimodules such that $\cC = \bigcup_{n}F_n(\cC)$, $\gr_n(\cC) = F_n(\cC)/F_{n-1}(\cC)$ is a projective $A$-bimodule and
\begin{equation}\label{eq:deltafilt}
\Delta_\cC(F_n(\cC)) \subseteq \sum_{i+j = n} F_i(\cC) \tensor{A} F_j(\cC)
\end{equation}
for all $n \geq 0$ (that is to say, $\Delta_\cC$ is a morphism of filtered $A$-bimodules). We will refer to these $A$-corings as \emph{graded projective filtered $A$-corings}. By convention, we put $F_{-1}(\cC) = 0$. Notice that the inclusion \eqref{eq:deltafilt} makes sense in view of the following well-known result for filtered (bi)modules (see, for instance, \cite[Lemma B.1]{LaiachiPaolo-complete}).

\begin{lemma}\label{lemma:locfgp}
Let $\cC$ be an $A$-bimodule with filtration $\left\{F_n(\cC) \mid n \in \N \right\}$ such that $\gr_n(\cC)$ is a projective $A$-bimodule for all $n \geq 0$.
Then
\begin{equation}\label{eq:FnProj}
F_n(\cC) \cong \bigoplus_{k=0}^n\frac{F_k(\cC)}{F_{k-1}(\cC)}
\end{equation}
for all $n\geq 0$ and so $F_n(\cC)$ is a projective $A$-bimodule. Moreover, the canonical map
\[\gr\left(\cC \right) \tensor{A} \gr\left(\cC \right) \to \gr\left(\cC  \tensor{A} \cC \right), \quad \big(\xi + F_h\left(\cC \right)\big) \tensor{A} \big(\xi' + F_{k}\left(\cC \right)\big) \mapsto \xi\tensor{A}\xi' + F_{h+k+1}\left(\cC  \tensor{A} \cC \right),\]
from \eqref{eq:laxmon} is an isomorphism of $A$-bimodules. If, in addition, the filtration $\left\{F_n(\cC) \mid n \in \N \right\}$ is exhaustive, then $\cC \cong \gr(\cC)$ as $A$-bimodules and, in particular, $\cC$ is a projective $A$-bimodule.
\end{lemma}

\begin{proof}
By definition and projectivity of $\gr_n(\cC)$, we have a split short exact sequence
\[
\xymatrix @C=15pt{
0 \ar[r] & F_{n-1}(\cC) \ar[r] & F_{n}(\cC) \ar[r] & \gr_n(\cC) \ar[r] & 0
}
\]
 of $A$-bimodules, which implies that, as $A$-bimodules, 
\[
F_n(\cC) \cong F_{n-1}(\cC) \oplus \gr_n(\cC).
\]
By proceeding recursively, one reaches \eqref{eq:FnProj}. The second claim follows from \cite[Theorem C.24, p. 93]{Majewski}. About the last claim in the statement, saying that the filtration is exhaustive means that $\cC \cong \injlimit{n}{F_n(\cC )}$ as $A$-bimodules. Since \eqref{eq:FnProj} means that $F_n(\cC )\cong F_n\left(\gr(\cC )\right)$ as $A$-bimodules, we have that $\cC \cong\injlimit{n}{F_n(\cC )}\cong \injlimit{n}{F_n\left(\gr(\cC )\right)}\cong\gr(\cC )$ as claimed.
\end{proof}
 
Analogously to the theory of filtered coalgebras (see, for example, \cite[\S 11.1]{Sweedler}), the graded $A$-bimodule $\gr(\cC)$ associated to a graded projective filtered $A$-coring $\cC$ becomes a graded $A$-coring in a natural way, as the subsequent Proposition \ref{prop:GrAss} formalize.

For the sake of clarity, by a \emph{graded $A$-coring} we mean an $A$-coring $\cD$ endowed with a graduation $\left\{\cD_n \mid n \in \N\right\}$ as $A$-bimodule such that every $\cD_n$ is projective as $A$-bimodule,
\[
\Delta_\cD\left(\cD_n\right) \subseteq \bigoplus_{i+j=n}\cD_i \tensor{A} \cD_j \qquad \text{for all } n \in \N \text{ and} \qquad \varepsilon_\cD\left(\cD_n\right) = 0 \qquad \text{for all } n\geq 1.  
\]
It can be seen as a comonoid in the monoidal category of graded $A$-bimodules. Notice that $\Delta_\cD$ is uniquely determined by the $A$-bilinear maps
\[\Delta_\cD^{[n]} : \cD_n \to \bigoplus_{i+j=n}\cD_i \tensor{A} \cD_j\]
obtained by (co)restriction of $\Delta_\cD$ to the graded components of $\cD$ and $\cD \tensor{A}\cD$ and which, in turn, are uniquely determined by the $A$-bilinear maps
\begin{equation}\label{eq:Deltacomps}
\Delta_\cD^{[h,k]} \coloneqq \left( \cD_n \xrightarrow{\Delta_{\cD}^{[n]}} \bigoplus_{i+j=n}\cD_i \tensor{A} \cD_j \xrightarrow{p_{h,k}^{\cD}} \cD_h \tensor{A} \cD_k\right)
\end{equation}
for all $n \geq 0$ and for all $h+k=n$. Following \cite[Definition 2.2]{ArdiMen}, we say that the graded $A$-coring $\cD$ is \emph{strongly graded} whenever $\Delta_\cD^{[h,k]}$ is injective for all $h,k \in \N$.

\begin{proposition}\label{prop:GrAss}
Let $\cC$ be a graded projective filtered $A$-coring. Then the $A$-coring structure on $\cC$ induces an $A$-coring structure on $\gr(\cC)$. Moreover, for any morphism of graded projective filtered $A$-corings $f : \cC \to \cB$, the induced graded $A$-bilinear morphism $\gr(f) : \gr(\cC) \to \gr(\cB)$ is a morphism of graded $A$-corings.
\end{proposition}

\begin{proof}
The first claim follows from the functoriality of $\gr(-)$ and from the fact that $\Delta_\cC: \cC  \to \cC  \tensor{A}\cC $ and $\varepsilon_\cC: \cC  \to A$ are filtered morphisms of $A$-bimodules. In fact, they induce graded morphisms of $A$-bimodules
\[\gr(\cC ) \xrightarrow{\gr\left(\Delta_\cC\right)} \gr(\cC  \tensor{A}\cC ) \cong \gr(\cC ) \tensor{A} \gr(\cC ) \qquad \text{and} \qquad \gr(\cC ) \xrightarrow{\gr\left(\varepsilon_\cC\right)}\gr(A) \cong A,\]
which provides an $A$-coring structure on $\gr(\cC )$, since $\gr(-)$ is lax monoidal. Concerning the second claim, it is enough to apply $\gr(-)$ to the diagrams expressing the comultiplicativity and counitality of $f$.
\end{proof}

\begin{remark} 
For any morphism $f : \cC \to \cB$ of filtered $A$-corings, we have that
\begin{equation}\label{eq:Deltahk}
\begin{gathered}
\xymatrix @C=30pt @R=18pt{
{\gr_n(\cC )} \ar[r]^-{\Delta^{[n]}_{\gr(\cC)}} \ar[d]_-{\gr_n(f)} & {\displaystyle\bigoplus_{i+j=n}\gr_i(\cC ) \tensor{A} \gr_j(\cC )}\ar[d]^-{\underset{i+j=n}{\bigoplus}\gr_i(f) \tensor{A} \gr_j(f)} \\
{\gr_n(\cB )} \ar[r]_-{\Delta^{[n]}_{\gr(\cB)}}  & {\displaystyle\bigoplus_{i+j=n}\gr_i(\cB ) \tensor{A} \gr_j(\cB )}
}
\end{gathered}
\end{equation}
commutes for all $n \geq 0$.
\end{remark}

\begin{corollary}[of Proposition \ref{prop:GrAss}]
The assignment $\cC \mapsto \gr(\cC)$ induces a functor $\gr(-)$ from the category of graded projective filtered $A$-corings to the category of graded $A$-corings.
\end{corollary}


\subsection{\texorpdfstring{$A$}{A}-bialgebroids}

Next, we recall the definition of a left bialgebroid. It can be considered as a revised version of the notion of a $\times_A$-bialgebra as it appears in \cite[Definition 4.3]{Schauenburg-bialgebras}. However, we prefer to mimic \cite{Lu} as presented in \cite[Definition 2.2]{BrzezinskiMilitaru} (in light of \cite[Theorem 3.1]{BrzezinskiMilitaru}, this is something we may do). 

\begin{definition}
A \emph{left bialgebroid} is the datum of
\begin{enumerate}[label=(B\arabic*),ref=B\arabic*]
\item a pair $\left(A,\cB \right)$ of $\K$-algebras, called \emph{the base algebra} and \emph{the total algebra} respectively,
\item a $\K$-algebra map $\eta_{\cB}: \env{A} \to \cB $, inducing a \emph{source} $s_{\cB}:A\to \cB $ and a \emph{target} $t_{\cB}:\op{A}\to \cB $ which are $\K$-algebra maps, and making of $\cB $ an $\env{A}$-bimodule,
\item an $A$-coring structure $\left(\cB ,\Delta_\cB,\varepsilon_\cB\right)$ on the $A$-bimodule ${}_{\eta}\cB =\sMto{\cB }$, 
\end{enumerate}
satisfying
\begin{enumerate}[resume*]
\item $\Delta_\cB$ takes values into $\cB  \tak{A} \cB $ and $\Delta_\cB: \cB  \to \cB \tak{A}\cB $ is a morphism of $\K$-algebras, where the algebra structure on $\cB \tak{A}\cB $ is given by the component-wise product
\begin{equation}\label{eq:prod}
(\xi \tensor{A} \zeta)(\xi' \tensor{A} \zeta') = \xi\xi'\tensor{A}\zeta\zeta',
\end{equation}
for all $\xi,\xi',\zeta,\zeta'\in \cB$,
\item\label{item:B5} $\varepsilon_\cB\big(\xi s_\cB\left(\varepsilon\left(\xi'\right)\right)\big) = \varepsilon_\cB\left(\xi\xi'\right) = \varepsilon_\cB\big(\xi t_\cB\left(\op{\varepsilon_\cB\left(\xi '\right)}\right)\big)$ for all $\xi,\xi'\in \cB $.
\item\label{item:B6} $\varepsilon_\cB(1_\cB )=1_A$.
\end{enumerate}
A $\K$-linear map $\varepsilon_\cB: \cB  \to A$ which is left $\env{A}$-linear and satisfies \eqref{item:B5} and \eqref{item:B6} is called a \emph{left character} on the $\env{A}$-ring $\cB $ (see \cite[Lemma 2.5 and following]{Bohm-handbook}).

A morphism of left bialgebroids from $(A,\cB)$ to $(A',\cB')$ is a pair of $\K$-algebra morphisms $\phi_0: A \to A'$ and $\phi_1 : \cB \to \cB'$ such that
\[
\phi_1 \circ s_{\cB} = s_{\cB'} \circ \phi_0, \qquad \phi_1 \circ t_{\cB} = t_{\cB'} \circ \phi_0, \qquad \varepsilon_{\cB'}\circ \phi_1 = \phi_0 \circ \varepsilon_{\cB} 
\]
and
\[
\xymatrix@C=15pt{
\cB \ar[rrr]^-{\phi_1} \ar[d]_-{\Delta_\cB} & & & \cB' \ar[d]^-{\Delta_{\cB'}}  \\
\cB \tensor{A} \cB \ar[rr]_-{\phi_1 \tensor{A} \phi_1} & & \cB'\tensor{A} \cB' \ar@{->>}[r] & \cB'\tensor{A'} \cB'
}
\]
commutes.
In this paper we will focus on left bialgebroids over a fixed base algebra $A$, that we call \emph{left $A$-bialgebroids}. A morphism of left $A$-bialgebroids between $\cB $ and $\cB ^{\prime }$ is then an algebra map $\phi :\cB \to \cB ^{\prime }$ such that
\begin{equation*}
\phi \circ s_{\cB }=s_{\cB ^{\prime }},\qquad \phi \circ t_{\cB }=t_{\cB ^{\prime }},\qquad \varepsilon _{\cB ^{\prime }}\circ \phi =\varepsilon _{\cB },\qquad \left( \phi \tensor{A}\phi \right) \circ \Delta _{\cB }=\Delta _{\cB ^{\prime }}\circ \phi .
\end{equation*}
The category of left $A$-bialgebroids will be denoted by $\Bialgd_{A}$.
\end{definition}

We will often omit to specify the $A$-bialgebroid $\cB$ in writing the comultiplication $\Delta_\cB$ or the counit $\varepsilon_\cB$, when it is clear from the context.

\begin{remark}
Let us make explicit some of the relations involved in the definition of a left bialgebroid and some of their consequences. In terms of elements of $A$ and $\cB $, and by resorting to Sweedler Sigma Notation, relations \eqref{eq:coasscoun} become
\begin{equation}\label{eq:coasscounel}
\begin{gathered}
\sum \xi_{(1)(1)} \tensor{A} \xi_{(1)(2)} \tensor{A} \xi_{(2)} = \sum \xi_{(1)}\tensor{A} \xi_{(2)(1)}\tensor{A} \xi_{(2)(2)} \\
 \text{and} \qquad \sum s\big(\varepsilon\left(\xi_{(1)}\right)\big)\xi_{(2)} = \xi = \sum t\big(\op{\varepsilon\left(\xi_{(2)}\right)}\big)\xi_{(1)}
\end{gathered}
\end{equation}
for all $\xi\in \cB $. Moreover, for all $a,b \in A$ the $A$-bilinearity of $\Delta$ forces
\[
\sum \big( s\left( a\right) t\left( \op{b}\right) \xi \big) _{\left( 1\right) }\tensor{A} \big( s\left( a\right) t\left( \op{b}\right) \xi \big) _{\left( 2\right) }
= \sum s\left( a\right) \xi _{\left( 1\right) }\tensor{A}t\left( \op{b}\right) \xi _{\left( 2\right) }.
\]
In particular, 
\begin{equation}\label{eq:DeltaLin}
\Delta\left(s\left(a\right)\right) = s\left(a\right)\tensor{A}1_\cB  \qquad \text{and} \qquad \Delta\left(t\left(\op{a}\right)\right) = 1_\cB  \tensor{A} t\left(\op{a}\right)
\end{equation}
for all $a\in A$. As a consequence, the multiplicativity of $\Delta$ forces
\begin{equation}\label{eq:deltablack}
\begin{gathered}
\Delta(\xi s(a)) = \Delta(\xi )\Delta(s(a)) = \big(\sum \xi _{(1)}\tensor{A}\xi _{(2)}\big)\big(s(a)\tensor{A}1_\cB \big) = \sum \xi _{(1)}s(a)\tensor{A}\xi _{(2)}, \\
\Delta(\xi t(\op{a})) = \Delta(\xi )\Delta(t(\op{a})) = \big(\sum \xi _{(1)}\tensor{A}\xi _{(2)}\big)\big(1_\cB \tensor{A}t(\op{a})\big) = \sum \xi _{(1)}\tensor{A}\xi _{(2)}t(\op{a}),
\end{gathered}
\end{equation}
for all $\xi\in \cB $. The $A$-bilinearity of $\varepsilon$ becomes
\begin{equation}\label{eq:epsibil}
\varepsilon \big( s\left( a\right) t\left( \op{b}\right) \xi\big) = a\varepsilon \left(\xi\right) b
\end{equation}
and since it also preserves the unit, we have that 
\begin{equation}
\varepsilon \left( s\left( a\right) \right) =\varepsilon \left(s(a) 1_\cB \right) =a\varepsilon \left( 1_\cB \right) = a =\varepsilon\left( 1_\cB \right) a = \varepsilon\left( t\left( \op{a}\right) \right) .  \label{eq:epsist}
\end{equation}
Therefore, in light of the character condition on $\varepsilon $,
\begin{equation}
\varepsilon \left( \xi s\left(a\right) \right) \stackrel{\eqref{item:B5}}{=} \varepsilon \bigg( \xi t\Big( \op{\varepsilon \big( s\left(a\right) \big)} \Big)\bigg) \stackrel{\eqref{eq:epsist} }{=} \varepsilon \left( \xi t\left( \op{a}\right) \right) .  \label{eq:epsiscambia}
\end{equation}
\end{remark}

Henceforth, all bialgebroids will be left ones, whence we will omit to specify it.

\begin{example}\label{ex:bialgebroids}
Let us give some examples.

\begin{enumerate}[leftmargin=0.7cm,label=(\alph*),ref=(\alph*)]
\item Any bialgebra over a field $\K$ is a $\K$-bialgebroid.

\item\label{item:enveloping} On the algebra $A\otimes \op{A}$ we may consider the morphisms
\begin{gather*}
s :A\to A\otimes \op{A},\quad a\mapsto a\otimes \op{1}, \qquad t:\op{A}\to A\otimes \op{A},\quad \op{b}\mapsto 1\otimes \op{b}, \\
\Delta :A\otimes \op{A}\to \left( A\otimes \op{A}\right) \otimes _{A}\left( A\otimes \op{A}\right) , \quad a\otimes \op{b}\mapsto \left( a\otimes \op{1}\right) \tensor{A}\left( 1\otimes \op{b}\right) , \\
\varepsilon :A\otimes \op{A}\to A, \quad a\otimes \op{b}\mapsto ab.
\end{gather*}
These make of $A\otimes \op{A}$ and $A$-bialgebroid (see \cite[Example 3.1]{Lu}).

\item\label{item:exb} Assume that $A$ is a finite-dimensional algebra over $\K$ and set $\cB \coloneqq \End{\K}{A}$. Consider $s:A \to \cB$ given by left multiplication and $t:A \to \cB$ given by right multiplication. The morphism
\[\cB \otimes \cB \to \Hom{}{}{\K}{}{A \otimes A}{A}, \qquad f \otimes g \mapsto \left[a \otimes b \mapsto f(a)g(b)\right]\]
induces an isomorphism $\cB \tensor{A} \cB \cong \Hom{}{}{\K}{}{A \otimes A}{A}$.
In view of this, one can endow $\cB$ with a structure of $A$-bialgebroid with source $s$, target $t$, $\Delta$ given by
\[\Delta(f)(a \otimes b) = f(ab)\]
(up to the latter isomorphism) and $\varepsilon$ by evaluation at $1_A$ (see \cite[page 56]{Lu}).

\item Let $R \to S$ be a depth two ring extension (see \cite[Definition 3.1]{KadSzla}) and set $A \coloneqq C_S(R)$, the centralizer of $R$ in $S$. Then the ring of endomorphisms $\cB \coloneqq \End{R}{S}$ of $S$ as an $R$-bimodule satisfies
\[\cB \tensor{A} \cB \cong \Hom{R}{}{R}{}{S \tensor{R} S}{S}, \qquad f \tensor{A} g \mapsto \left[s \tensor{R}s'\mapsto f(s)g(s')\right],\]
as above (see \cite[Proposition 3.11]{KadSzla}) and we may endow it with an $A$-bialgebroid structure exactly as in \ref{item:exb} (\cite[Theorem 4.1]{KadSzla}).

\item Let $\left( H,m,u,\Delta ,\varepsilon \right) $ be a bialgebra and let $A$ be a braided commutative algebra in ${_H^H\mathcal{YD}}$. This means that $A$ is at the same time a left $H$-module algebra (that is, an algebra in the monoidal category $\left(\Lmod{H},\otimes ,\K \right) $ of left $H$-modules) with action $H\otimes A\to A, h \otimes a\mapsto h\cdot a$, satisfying
\[
h\cdot \left( ab\right) = \sum \left( h_{\left( 1\right) }\cdot a\right) \left(h_{\left( 2\right) }\cdot b\right) \qquad \text{and}\qquad h\cdot 1_{A}=\varepsilon \left( h\right) 1_{A} 
\]
for all $h\in H$, $a,b\in A$, and a left $H$-comodule algebra with coaction $\rho:A \to H \otimes A$, satisfying
\[
\rho(ab) = \sum a_{(-1)}b_{(-1)} \otimes a_{(0)}b_{(0)} \qquad \text{and} \qquad \rho(1_A) = 1_H \otimes 1_A
\]
for all $a,b \in A$. Furthermore, these structures are required to satisfy
\[\sum h_{(1)}a_{(-1)} \otimes h_{(2)}\cdot a_{(0)} = \sum \left(h_{(1)}\cdot a\right)_{(-1)}h_{(2)} \otimes \left(h_{(1)}\cdot a\right)_{(0)}\]
and $\sum \left(a_{(-1)}\cdot b\right)a_{(0)} = ab$ for all $a,b \in A$ and $h \in H$ (the latter expresses the braided commutativity). Under these conditions, the smash product algebra $H \# A$ is an $A$-bialgebroid with 
\begin{gather*}
(h \otimes a)(h' \otimes b) = \sum h_{(1)}h' \otimes \left(h_{(2)}\cdot b\right)a, \qquad 1_{H \# A} = 1_H \otimes 1_A, \\
 s(a) = \sum a_{(-1)} \otimes a_{(0)}, \qquad t(\op{a}) = 1_H \otimes a, \\
\Delta(h \otimes a) = \sum \left( h_{(1)} \otimes 1_A \right) \tensor{A} \left( h_{(2)} \otimes a\right) \qquad \text{and} \qquad \varepsilon(h \otimes a) = \varepsilon(h)a
\end{gather*}
for all $a,b \in A$, $h,h'\in H$. This is a left-left symmetrical version of \cite[Example 3.4.7]{Bohm-handbook} and \cite[Theorem 4.1]{BrzezinskiMilitaru}.

\item \emph{Connes-Moscovici's bialgebroid} (see \cite{ConnesMoscovici,PanaiteOystaeyen} and \cite[Example 3.4.6]{Bohm-handbook}). Let $H$ be a Hopf algebra (in fact, a bialgebra would suffice) and let $A$ be an $H$-module algebra. The vector space $\cB \coloneqq A\otimes H\otimes A$ becomes an algebra via
\begin{equation}\label{eq:CMprod}
\left( a\otimes h\otimes b\right) \left( a^{\prime }\otimes h^{\prime}\otimes b^{\prime }\right) =\sum a\left( h_{\left( 1\right) }\cdot a^{\prime}\right) \otimes h_{\left( 2\right) }h^{\prime }\otimes \left( h_{\left(3\right) }\cdot b^{\prime }\right) b 
\end{equation}
and unit $1_{A}\otimes 1_{H}\otimes 1_{A}$. It can be endowed with an $A$-bialgebroid structure as follows
\begin{gather*}
s_\cB (a) \coloneqq a \otimes 1_H \otimes 1_A, \qquad t_\cB (\op{a}) \coloneqq 1_A \otimes 1_H \otimes a, \notag \\
\Delta _{\cB }\left( a\otimes h\otimes b\right)  \coloneqq \sum \left( a\otimes h_{\left( 1\right) }\otimes 1_{A}\right) \otimes _{A}\left( 1_{A}\otimes h_{\left(2\right) }\otimes b\right) , \\
\varepsilon _{\cB }\left( a\otimes h\otimes b\right) \coloneqq a\varepsilon_{H}\left( h\right) b. 
\end{gather*}
Following \cite{PanaiteOystaeyen}, we will denote this bialgebroid by $A\odot H\odot A$, shunning the use of symbols like $\# $ or $\ltimes,\rtimes$ in order to avoid confusion with two-sided smash/crossed products in the sense of \cite{HausserNill-diag}.
Notice that for all $a,a^{\prime },b,b^{\prime }\in A$ and $h\in H$ we have
\begin{align*}
a\triangleright \left( a^{\prime }\otimes h\otimes b^{\prime }\right)\triangleleft b & \stackrel{\eqref{eq:triangles}}{=} \left( a\otimes 1_{H}\otimes 1_{A}\right) \left(1_{A}\otimes 1_{H}\otimes b\right) \left( a^{\prime }\otimes h\otimes b^{\prime }\right) \stackrel{\eqref{eq:CMprod}}{=} aa^{\prime }\otimes h\otimes b^{\prime }b, \\
a\blacktriangleright \left( a^{\prime }\otimes h\otimes b^{\prime }\right)\blacktriangleleft b & \stackrel{\eqref{eq:triangles}}{=} \left( a^{\prime }\otimes h\otimes b^{\prime}\right) \left( b\otimes 1_{H}\otimes 1_{A}\right) \left( 1_{A}\otimes 1_{H}\otimes a\right)  \\
 & \stackrel{\eqref{eq:CMprod}}{=} \sum a^{\prime }\left( h_{\left( 1\right) }\cdot b\right) \otimes h_{\left( 2\right) }\otimes \left( h_{\left( 3\right) }\cdot a\right) b^{\prime } .
\end{align*}
\end{enumerate}
\end{example}

\begin{remark}
It is known that an $A^{e}$-ring is a left bialgebroid if and only if the category $\Lmod{\cB}$ of left $\cB $-modules is monoidal in such a way that the forgetful functor $\Lmod{\cB} \to \Lmod{\env{A}}$ is a monoidal functor (see, for example, \cite[Theorem 5.1]{Schauenburg-bialgebras}). In particular, the tensor product of two $\cB $-modules $M$ and $N$ is $M\tensor{A}N$ with diagonal action $\xi \centerdot \left( m\tensor{A}n\right) \coloneqq \sum \left( \xi _{\left( 1\right) }\centerdot m\right) \tensor{A}\left( \xi _{\left( 2\right) }\centerdot n\right) $ and $A$ is a $\cB $-module with 
\begin{equation}\label{eq:dotaction}
\xi \centerdot a \coloneqq \varepsilon \left( \xi t\left( \op{a}\right) \right) \stackrel{\eqref{eq:epsiscambia}}{=} \varepsilon
\left( \xi s\left( a\right) \right) .  
\end{equation}
It is, in fact, a left $\cB$-module algebra (see \cite[\S3.7.1]{Bohm-handbook} for a right-handed analogue).
In particular, for all $a,b \in A$ and all $\xi \in \cB$ we have
\begin{equation}
\sum\left( \xi _{\left( 1\right) }\centerdot a\right) \left( \xi _{\left( 2\right)
}\centerdot b\right) =\xi \centerdot ab \qquad \text{and} \qquad \xi\centerdot 1_A = \varepsilon(\xi).  \label{eq:modulealgebra}
\end{equation}
\end{remark}


\section{The Connes-Moscovici's bialgebroid as universal \texorpdfstring{$\env{A}$}{Ae}-ring}\label{sec:UEA}

In this section we introduce $A$-anchored Lie algebras and we show that the Connes-Moscovici's bialgebroid $A \odot U_\K(L) \odot A$ naturally associated to an $A$-anchored Lie algebra satisfies a universal property as $\env{A}$-ring. In particular, unless stated otherwise, we assume to work over a fixed base algebra $A$, possibly non-commutative. We conclude the section with an extension of the PBW theorem to bialgebroids of the form $A \odot U_\K(L) \odot A$.


\subsection{\texorpdfstring{$A$}{A}-anchored Lie algebras}\label{ssec:Aanch}

\begin{definition}\label{eq:Aanch}
We call $A$\emph{-anchored Lie algebra} an ordinary Lie algebra $L$ over $\K$ together with a Lie algebra morphism $\omega :L\to \Derk{A}$, called the \emph{anchor}.
We will often write $X \cdot a$ for $\omega(X)(a)$. A morphism of $A$-anchored Lie algebras between $\left( L,\omega \right) $ and $\left( L^{\prime },\omega ^{\prime }\right) $ is a Lie algebra morphism $f:L\to L^{\prime }$ such that $\omega^{\prime }\circ f=\omega $. The category of $A$-anchored Lie algebras and their morphisms will be denoted by $\ALie{A}$.
\end{definition}

\begin{remark}
The reader needs to be warned that the terminology ``$A$-anchored'' used here is inspired from the literature, but it neither strictly coincides with the classical notion of $A$-anchored module, nor it properly extends it. In fact, in the literature, an ``$A$-anchored module'' \cite[\S1]{Popescu} (also called ``$A$-module with arrow'' \cite[\S3]{Popescu1} or ``$A$-module fl{\'e}ch{\'e}'' \cite[\S1]{Pradines}) is an $A$-module $M$ over a \emph{commutative} algebra $A$ together with an \emph{$A$-linear map} $M \to \Derk{A}$. Since, in the present framework, $A$ is assumed to be preferably non-commutative, the vector space $\Derk{A}$ does not carry any natural $A$-module structure and hence we do not have any reasonable way to speak about an $A$-linear anchor. In spite of this, in order to limit the proliferation of different terminology in the field and trusting that the non-commutative context will help in distinguishing between the two notions, we decided to adopt the term ``$A$-anchored'' in this framework as well.
\end{remark}

\begin{example}\label{Ex:anchored}
Let us give a few important examples.
\begin{enumerate}[leftmargin=0.7cm,label=(\alph*),ref=(\alph*)]
\item The Lie algebra $\Derk{A}$ with the identity map is an $A$-anchored Lie algebra.

\item Any Lie algebra $L$ is a $U_\K(L)$-anchored Lie algebra.

\item Any $\K$-algebra $A$, with the associated Lie algebra structure $\cL(A) = \left(A,[-,-]\right)$ given by the commutator bracket, is an $A$-anchored Lie algebra with anchor induced by the adjoint representation. Namely,
\[\cL(A) \to \Derk{A}, \qquad a \mapsto \big[b \mapsto [a,b]\big].\]

\item A \emph{Lie-Rinehart algebra} over a commutative algebra $R$ (called in this way in honour of G.~S.~Rinehart, who studied them in \cite{Rinehart} under the name of $(K,R)$-Lie algebras) is a Lie algebra $L$ endowed with a (left) $R$-module structure $R\otimes L\to L, r \otimes X \mapsto r \cdot X,$ and with a Lie algebra morphism $\omega:L\to \Derk{R} $ such that, for all $r \in R$ and $X,Y\in L$,
\[
\omega \left( r\cdot X\right) = r\cdot \omega \left( X\right) \qquad \text{and}\qquad \left[ X,r\cdot Y\right] =r\cdot \left[ X,Y\right] +\omega \left(X\right) \left( r\right) \cdot Y.
\]
Clearly, any Lie-Rinehart algebra over $R$ is an $R$-anchored Lie algebra.

\item\label{ex:itemPrim} Let $\cB $ be an $A$-bialgebroid and consider the vector space of \emph{primitive elements}
\begin{equation*}
\prim \left( \cB \right) \coloneqq \big\{ X \in \cB \mid \Delta \left( X \right) = X \tensor{A}1+1\tensor{A}X \big\} .
\end{equation*}
This is a Lie algebra with the commutator bracket.  
Assume that $X\in \prim \left( \cB \right) $. In light of Equation \eqref{eq:modulealgebra}, $X$ acts on $A$ by derivations, which means that the assignment
\[
\omega _{\cB }:\prim \left( \cB \right) \to \Derk{A} , \qquad X\mapsto X\centerdot (-),
\]
is well-defined. Moreover,
\begin{align*}
\omega _{\cB }\left( XY-YX\right) & \left( a\right) = \left( XY\right) \centerdot a-\left( YX\right) \centerdot a = X\centerdot \left( Y\centerdot a\right) -Y\centerdot \left( X\centerdot a\right) \\
&= \omega _{\cB }\left( X\right) \left( \omega _{\cB }\left( Y\right) \left( a\right) \right) -\omega _{\cB }\left( Y\right) \left( \omega _{\cB }\left( X\right) \left( a\right) \right) = \left[ \omega _{\cB }\left( X\right) ,\omega _{\cB }\left( Y\right) \right] \left( a\right)
\end{align*}
implies that $\omega _{\cB }$ is a morphism of Lie algebras and hence it is an
anchor for $\prim \left( \cB \right) $.
As a matter of notation, we write $\theta_\cB: \prim(\cB) \to \cB$ for the canonical inclusion.
\end{enumerate}
\end{example}

\begin{definition}\label{def:manydefs}
Assume that $(L,\omega)$ is an $A$-anchored Lie algebra. 
\begin{enumerate}[label=(L\arabic*)]
\item An \emph{$A$-anchored Lie ideal} $(L',\omega')$ in $(L,\omega)$ is a Lie ideal $L'$ in $L$ together with an anchor $\omega'$ such that the inclusion $L'\subseteq L$ is a morphism in $\ALie{A}$. Equivalently, it is vector subspace $L'\subseteq L$ such that $[X,X'] \in L'$ for all $X'\in L', X \in L$ with anchor $\omega'$ given by the restriction of $\omega$.
\item An \emph{$A$-anchored Lie subalgebra} $(L'',\omega'')$ in $(L,\omega)$ is a Lie subalgebra $L''$ of $L$ together with an anchor $\omega''$ such that the inclusion $L''\subseteq L$ is a morphism in $\ALie{A}$. Equivalently, it is vector subspace $L''\subseteq L$ such that $[X'',Y''] \in L''$ for all $X'',Y''\in L''$ with anchor $\omega''$ given by the restriction of $\omega$.
\item\label{item:L3} If we have two $A$-anchored Lie algebras $(L',\omega')$ and $(L'',\omega'')$ and a Lie algebra morphism $\delta:L'' \to \Derk{L'}$ such that
\begin{equation}\label{eq:deltaomega}
\Big[\omega''(X''),\omega'(X')\Big] = \omega'\big(\delta(X'')(X')\big)
\end{equation}
in $\Derk{A}$ for all $X'\in L'$ and all $X'' \in L''$, then we define the \emph{semi-direct product} of $(L',\omega')$ and $(L'',\omega'')$ to be the $\K$-vector space $L''\times L'$ with Lie bracket
\begin{equation}\label{eq:SDPbracket}
\big[(X'', X'), (Y'', Y') \big] \coloneqq \big([X'',Y''], \delta(X'')(Y') - \delta(Y'')(X') + [X',Y']\big)
\end{equation}
and anchor
\begin{equation}\label{eq:SDPanchor}
L \to \Derk{A}, \qquad (X'',X') \mapsto \omega''(X'') + \omega'(X').
\end{equation}
We denote it by $\left(L'' \ltimes_\delta L',\omega_\delta\right)$.
\end{enumerate}
\end{definition}

For the sake of brevity, from now on we will only speak about ideals and subalgebras without reporting the syntagma ``$A$-anchored Lie'' in front. Definition \ref{def:manydefs}\ref{item:L3} is consistent in view of the following results.

\begin{lemma}\label{lem:semidirect}
The semi-direct product $(L,\omega)$ of two $A$-anchored Lie algebras $(L',\omega')$ and $(L'',\omega'')$ is an $A$-anchored Lie algebra.
\end{lemma}

\begin{proof}
The fact that the semi-direct product is a Lie algebra follows from the fact that, as Lie algebras, it is the semi-direct product of $L'$ and $L''$. 
Thus we only need to check that $\omega_\delta$ is a Lie algebra morphism. To this aim, we compute directly
\begin{align*}
& \Big[\omega_\delta\big((X'', X')\big), \omega_\delta\big((Y'', Y') \big) \Big] \stackrel{\eqref{eq:SDPanchor}}{=} \Big[ \omega''(X'') + \omega'(X'), \omega''(Y'') + \omega'(Y') \Big] \\
  & \stackrel{\phantom{(45)}}{=} \Big[ \omega''(X''), \omega''(Y'') \Big] + \Big[ \omega'(X'), \omega''(Y'') \Big] + \Big[ \omega''(X''),  \omega'(Y') \Big] + \Big[  \omega'(X'),  \omega'(Y') \Big] \\
	& \stackrel{\eqref{eq:deltaomega}}{=} \omega''\big([X'',Y'']\big) + \omega'\big(\delta(X'')(Y')\big) - \omega'\big(\delta(Y'')(X')\big) + \omega'\big([X',Y']\big) \\
	& \stackrel{\eqref{eq:SDPanchor}}{=} \omega_\delta\Big(\big([X'',Y''], \delta(X'')(Y') - \delta(Y'')(X') + [X',Y']\big)\Big) \stackrel{\eqref{eq:SDPbracket}}{=} \omega_\delta\Big(\big[(X'', X'), (Y'', Y') \big] \Big)
\end{align*}
for all $X',Y'\in L'$ and all $X'',Y'' \in L''$.
\end{proof}

The following lemma should not be surprising.

\begin{lemma}\label{lem:sdprod}
Let $(L,\omega)$ be an $A$-anchored Lie algebra and let $(L',\omega')$ and $(L'',\omega'')$ be subalgebras of $(L,\omega)$. Then there exists a semi-direct product $\left(L'' \ltimes_\delta L',\omega_\delta\right)$ of $(L',\omega')$ and $(L'',\omega'')$ such that $L'' \ltimes_\delta L' \to L, (X'',X')\mapsto X'' + X'$, is an isomorphism of $A$-anchored Lie algebras if and only if
\begin{itemize}[leftmargin=0.7cm]
\item $(L',\omega')$ is an ideal in $(L,\omega)$;
\item $L = L'' + L'$ as $\K$-vector spaces;
\item $L'\cap L'' = 0$ (that is, $L = L'' \oplus L'$).
\end{itemize}
If this is the case, then $\delta: L'' \to \Derk{L'}$ is given by $X'' \mapsto [X'',-]$.
\end{lemma}

\begin{proof}
In one direction, notice that $(L',\omega')$ is an ideal in $\left(L'' \ltimes_\delta L',\omega_\delta\right)$ via the canonical morphism $L'\to L'' \oplus L', X'\mapsto (0,X'),$ and that $(L'',\omega'')$ is a subalgebra of $\left(L'' \ltimes_\delta L',\omega_\delta\right)$ via the canonical morphism $L''\to L'' \oplus L', X''\mapsto (X'',0)$. Furthermore, one recovers $\delta$ as
\[\big[(X'',0),(0,X')\big] = \big(0,\delta(X'')(X')\big)\]
for all $X'\in L', X'' \in L''$. In the other direction, assume that $(L',\omega')$ is an ideal in $(L,\omega)$ and that $L = L'' \oplus L'$ as $\K$-vector spaces. Consider further the assignment $\delta: L'' \to \Derk{L'}$ given by $X'' \mapsto [X'',-]_L$ (which is well-defined because $L'$ is an ideal). It satisfies
\begin{align*}
\big[\omega''(X''),\omega'(X')\big] & = \omega(X'')\circ \omega(X') - \omega(X') \circ \omega(X'' ) \\
 & = \omega\big([X'',X']_L\big) = \omega\big(\delta(X'')(X')\big) = \omega'\big(\delta(X'')(X')\big)
\end{align*}
for all $X'\in L',X'' \in L'',$ which is \eqref{eq:deltaomega}, and so we may perform the semi-direct product $\left(L'' \ltimes_\delta L',\omega_\delta\right)$.
 Then
\begin{align*}
\big[X'' + X', Y'' + Y'\big]_L & = [X'', Y'']_L + [X', Y'' ]_L + [X'' ,  Y']_L + [ X',  Y']_L \\ 
 & = [X'', Y'']_{L''} + \big(\delta(X'')(Y') - \delta(Y'')(X') + [ X',  Y']_{L'}\big),
\end{align*}
so that $L'' \oplus L' \cong L$ is a morphism of Lie algebras and moreover
\[
\omega(X'' + X') = \omega(X'') + \omega(X') = \omega''(X'') + \omega'(X'),
\]
whence it is of $A$-anchored Lie algebras, too.
\end{proof}

\begin{remark}\label{rem:improper}
The reader has to be warned that, despite the definition of ideal and of semi-direct product in Definition \ref{def:manydefs} has been inspired by the subsequent Lemma \ref{lem:LieIdeal} and Proposition \ref{prop:primitive} and by the results in \S\ref{ssec:commutative}, they may turn out to be improper terminologies in the future. In fact, it is not true in general that the quotient of an $A$-anchored Lie algebra by an ideal is an $A$-anchored Lie algebra (unless the ideal has the zero anchor) or that $(L,\omega)$ is a semi-direct product of $(L',\omega')$ and $(L'',\omega'')$ if and only if there is a short exact sequence of $A$-anchored Lie algebras
\begin{equation}\label{eq:ses}
0 \to L' \xrightarrow{f} L \xrightarrow{g} L'' \to 0
\end{equation}
such that $g$ admits a section $\sigma$ which is a morphism of $A$-anchored Lie algebras. On the one hand, the canonical projection $L'' \ltimes_{\delta} L' \to L''$ is not a morphism of $A$-anchored Lie algebra because it is not compatible with the anchors. On the other hand, in order to have that $\omega'' \circ g = \omega$ and that $\omega \circ f = \omega '$, we should have had that
\[\omega'= \omega \circ f = \omega'' \circ g \circ f = 0,\]
which is not the case in general. 

What one may observe is that $(L,\omega)$ is a semi-direct product of $(L',\omega')$ and $(L'',\omega'')$ if and only if there is a short exact sequence of Lie algebras \eqref{eq:ses} such that $g$ admits a section $\sigma$ and both $f$ and $\sigma$ are morphisms of $A$-anchored Lie algebras (but $g$ is not, in general).
\end{remark}


\subsection{A universal \texorpdfstring{$\env{A}$}{Ae}-ring construction}\label{ssec:UEA}

Assume that we are given an $A$-anchored Lie algebra $\left( L,\omega \right) $. Recall that we may consider the universal enveloping algebra $U_{\K}\left( L\right)$ of $L$ and that there is a canonical injective $\K$-linear map
\begin{equation}\label{eq:jL}
j_L:L \to U_\K(L), \qquad X \mapsto x,
\end{equation} 
which allows us to identify $X$ with its image $x$ in $U_\K(L)$. 
The anchor $\omega$ makes of $A$ a left representation of $L$ with $L$ acting as derivations, that is, we have a Lie algebra morphism
\begin{equation}\label{eq:omegaXder}
L\xrightarrow{\omega } \Derk{A} \subseteq \cL\left( \End{\K}{ A} \right)
\end{equation}
where $\cL\left( \End{\K}{A} \right) $ is the Lie algebra associated to the associative algebra $\left( \End{\K}{A} ,\circ ,\mathrm{id}_{A}\right) $. By the universal property of the universal enveloping algebra, there is a unique algebra morphism 
\[
\Omega :U_{\K}\left( L\right) \to \mathrm{End}_{\K}\left(A\right)  
\]
which extends $\omega $. 

\begin{lemma}\label{lem:modulealgebra}
The base algebra $A$ is naturally an $U_{\K }\left( L\right) $-module algebra.
\end{lemma}

\begin{proof}
See, for instance, \cite[Example 6.1.13(3)]{dascalescu}. Explicitly
\begin{equation}\label{eq:Uaction}
x \cdot a = \omega(X)(a), \qquad x \cdot 1_A \stackrel{\eqref{eq:omegaXder}}{=} 0 \qquad \text{and} \qquad u\cdot a = \Omega(u)(a)
\end{equation}
for all $X \in L$, $u \in U_\K(L)$ and $a\in A$.
\end{proof}

By Lemma \ref{lem:modulealgebra}, we may consider the Connes-Moscovici's $A$-bialgebroid $A\odot U_{\K}\left( L\right) \odot A$. For the sake of simplicity, we will often denote it by $\cB_L$. As an $\env{A}$-ring, it comes endowed with a Lie algebra morphism
\[
J_L \colon L \to \cL\left(\cB_L\right), \qquad X \mapsto 1 \otimes x \otimes 1,
\] 
which satisfies
\begin{equation}\label{eq:JL2}
\Big[J_L(X),\eta_{\cB_L}(a \otimes \op{b})\Big] = \eta_{\cB_L}(X\cdot(a \otimes \op{b}))
\end{equation}
for all $a,b \in A$ and all $X \in L$, where
\[X\cdot\left(a \otimes \op{b}\right) = \left(X\cdot a\right) \otimes \op{b} + a \otimes \op{\left(X\cdot b\right)}\]
is the $L$-module structure on the tensor product of two $L$-modules. Equivalently, $\eta_{\cB_L} : \env{A} \to \cB_L$ is a morphism of $L$-modules, where $\cB_L$ has the $L$-module structure induced by $J_L$. 

The $\env{A}$-ring $\cB_L$ with $J_L$ is universal among pairs $(\cR,\phi_L)$ satisfying these properties.

\begin{theorem}\label{thm:UEA}
Given an $\env{A}$-ring $\cR$ with $\K$-algebra morphism $\phi_A : \env{A} \to \cR$ and given a Lie algebra morphism $\phi_L:L \to \cL(\cR)$ such that
\begin{equation}\label{eq:Llin}
\big[\phi_L(X),\phi_A(a \otimes \op{b})\big] = \phi_A\left(X \cdot (a \otimes \op{b})\right),
\end{equation}
for all $a,b \in A$ and all $X \in L$, there exists a unique morphism of $\env{A}$-rings $\Phi:\cB_L \to \cR$ such that $\Phi\circ J_L = \phi_L$. It is explicitly given by
\begin{equation}\label{eq:defPhi}
\Phi:\cB_L \to \cR, \qquad a \otimes u \otimes b \mapsto \phi_A(a \otimes \op{b})\phi'(u),
\end{equation}
where $\phi': U_\K(L) \to \cR$ is the unique morphism of $\K$-algebras such that $\phi'\circ j_L = \phi_L$.
\end{theorem}

\begin{proof}
By the universal property of the universal enveloping $\K$-algebra $U_\K(L)$, there exists a unique morphism of $\K$-algebras $\phi': U_\K(L) \to \cR$ such that $\phi'\circ j_L = \phi_L$. Now, set $U \coloneqq U_\K(L)$ and consider the $\K$-linear map
\[
\Phi:A \odot  U \odot  A \to \cR, \qquad a \otimes u \otimes b \mapsto \phi_A(a \otimes \op{b})\phi'(u)
\]
of \eqref{eq:defPhi}.
It follows immediately from the definition that
\[\Phi \circ \eta_{\cB_L} = \phi_A \qquad \text{and} \qquad \Phi\circ J_L = \phi_L.\]
Now, a straightforward check using \eqref{eq:Llin} and induction on a PBW basis of $U$ shows that
\begin{equation}\label{eq:techUEA}
\phi'(u) \phi_A(a \otimes \op{b}) = \sum\phi_A\left((u_{(1)}\cdot a)\otimes \op{(u_{(3)}\cdot b)}\right)\phi'\left(u_{(2)}\right)
\end{equation}
for all $u \in U$ and $a,b \in A$. In view of this, we have
\begin{align*}
\Phi\left((a \otimes u \otimes b)(a'\otimes v \otimes b')\right) & \stackrel{\eqref{eq:CMprod}}{=} \sum\Phi\left(a(u_{(1)}\cdot a') \otimes u_{(2)}v \otimes (u_{(3)}\cdot b')b\right) \\
 & \stackrel{\eqref{eq:defPhi}}{=} \sum \phi_A\left(a(u_{(1)}\cdot a') \otimes \op{\left((u_{(3)}\cdot b')b\right)}\right)\phi'\left(u_{(2)}v\right) \\
 & \stackrel{\phantom{(35)}}{=} \sum \phi_A\left(a \otimes \op{b}\right)\phi_A\left((u_{(1)}\cdot a')\otimes \op{(u_{(3)}\cdot b')}\right)\phi'\left(u_{(2)}\right)\phi'\left(v\right) \\
 & \stackrel{\eqref{eq:techUEA}}{=} \phi_A\left(a \otimes \op{b}\right)\phi'\left(u\right)\phi_A\left( a'\otimes \op{{b'}}\right)\phi'\left(v\right) \\
 & \stackrel{\phantom{(35)}}{=} \Phi\left(a \otimes u \otimes b\right)\Phi\left(a'\otimes v \otimes b'\right)
\end{align*}
for all $u,v \in U$, $a,a',b,b'\in A$. Thus, $\Phi$ is a morphism of $\env{A}$-rings and it is clearly the unique satisfying $\Phi\circ J_L = \phi_L$.
\end{proof}

\begin{corollary}\label{cor:representations}
Given an $A$-anchored Lie algebra $(L,\omega)$, any representation $\rho:L \to \End{\K}{M}$ of $L$ into an $A$-bimodule $M$ satisfying the Leibniz condition
\begin{equation}\label{eq:Leibniz}
\rho(X)(a\cdot m \cdot b) = a\cdot \rho(X)(m) \cdot b + \omega(X)(a)\cdot m \cdot b + a \cdot m \cdot \omega(X)(b)
\end{equation}
for all $a,b \in A$, $X \in L$ and $ m\in M$, makes of $M$ a left $A \odot  U_\K(L) \odot  A$-module, and conversely.
\end{corollary}

\begin{proof}
If $M$ is an $A$-bimodule, then the assignments
\[s_E:A \to \End{\K}{M}, \ a \mapsto [m \mapsto a\cdot m], \qquad \text{and}\qquad t_E:\op{A} \to \End{\K}{M}, \ \op{b} \mapsto [m \mapsto m\cdot b],\]
make of $\End{\K}{M}$ an $\env{A}$-ring with $\phi_A: \env{A} \to \End{\K}{M}, a\otimes \op{b} \mapsto s_E(a)t_E(\op{b})$. If we consider the Lie algebra morphism $\phi_L \coloneqq \rho$, then equation \eqref{eq:Leibniz} is exactly condition \eqref{eq:Llin} and hence there is a unique morphism of $\env{A}$-rings $R:A \odot  U_\K(L) \odot   A \to \End{\K}{M}$ such that $R \circ J_L = \rho$. The other way around, if we have a morphism of $\env{A}$-rings $R : A \odot  U_\K(L) \odot   A \to \End{\K}{M}$ and we compose it with $J_L$ we get a Lie algebra morphism $\phi_L : L \to \End{\K}{M}$ such that
\begin{align*}
\phi_L(X)(a\cdot m\cdot b) & \stackrel{\phantom{(20)}}{=} \left(R(1_A\otimes x\otimes 1_A) \circ \phi_A(a \otimes \op{b})\right)(m) \\
 & \stackrel{\phantom{(20)}}{=} R\left((1_A\otimes x\otimes 1_A)(a \otimes 1_U \otimes b)\right) (m) \\
 & \stackrel{\eqref{eq:CMprod}}{=} R\left((a\otimes x\otimes b) + (\omega(X)(a) \otimes 1_U \otimes b) + (a \otimes 1_U \otimes \omega(X)(b))\right) (m) \\ 
 & \stackrel{\phantom{(20)}}{=} a\cdot \phi_L(X)(m) \cdot b + \omega(X)(a)\cdot m \cdot b + a \cdot m \cdot \omega(X)(b). \qedhere
\end{align*}
\end{proof}

\begin{remark}\label{rem:PanaiteOystaeyen}
Observe that the algebra maps $\phi':U_\K(L) \to \cR$ and $\phi_A:\env{A} \to \cR$ satisfy \eqref{eq:techUEA} if and only if they satisfy
\[\sum \phi'\left(u_{(1)}\right)\phi_A\Big(a \otimes \op{\big(S\left(u_{(2)}\right) \cdot b\big)}\Big) = \sum \phi_A(u_{(1)} \cdot a \otimes \op{b})\phi'(u_{(2)})\]
for all $u \in U$, $a,b \in A$. This implies that the $\env{A}$-ring morphism $\Phi:A\odot  U_\K(L)\odot  A \to \cR$ in Theorem \ref{thm:UEA} coincides, up to the isomorphism of $A$-bialgebroids from \cite[Theorem 2.5]{PanaiteOystaeyen}, with the algebra map $\omega$ of \cite[Proposition 3.1]{PanaiteOystaeyen}.
\end{remark}

\begin{proposition}\label{prop:PBW}
Let $\mb{B} \coloneqq \left\{X_\alpha\mid \alpha \in S\right\}$ be a $\K$-basis for $L$, where $S$ is an ordered set of indexes. Then $\cB_L$ is a free left $\env{A}$-module with basis given by $1_A \otimes 1_U\otimes 1_A$ and $1_A \otimes x_{\alpha_1}\cdots x_{\alpha_n} \otimes 1_A$ where the $x_{\alpha_1}\cdots x_{\alpha_n}$'s are the cosets of the standard monomials $X_{\alpha_1}\cdots X_{\alpha_n}$ in the basis $\mb{B}$ (see \cite[\S V.2]{Jacobson}).
\end{proposition}

\begin{proof}
It follows from the PBW theorem (see, for instance, \cite[Theorem V.2.3]{Jacobson}) and the definition of the left $\env{A}$-module structure on $\cB_L$.
\end{proof}


\section[Connes-Moscovici's bialgebroid as universal enveloping bialgebroid]{\for{toc}{The Connes-Moscovici's bialgebroid as universal enveloping bialgebroid}\except{toc}{The Connes-Moscovici's bialgebroid as universal enveloping bialgebroid}}\label{sec:MM}

Our next aim is to prove that the Connes-Moscovici's bialgebroid $\cB_L=A \odot U_\K(L) \odot A$ satisfies a universal property as $A$-bialgebroid as well, in the form of an adjunction between the category of $A$-anchored Lie algebras and the category of $A$-bialgebroids.


\subsection{The primitive functor}\label{ssec:primitives}

In light of Example \ref{Ex:anchored}\ref{ex:itemPrim}, we may consider the assignment
\begin{equation*}
\mathbb{P}:\Bialgd_{A}\to \ALie{A}
\end{equation*}
given on objects by $\mathbb{P}\left( \cB \right) =\left( \prim \left(
\cB \right) ,\omega _{\cB }\right) $ and on morphisms by simply (co)restricting any $\phi :\cB \to \cB ^{\prime }$ to the primitive elements, that is $\phi \circ \theta_\cB = \theta_{\cB'}\circ \bP(\phi)$.
The latter gives a well-defined morphism of $A$-anchored Lie algebras because it is compatible with the commutator bracket 
and
\begin{align*}
\omega _{\cB ^{\prime }}\left( \phi \left( X\right) \right) \left( a\right) & = \varepsilon _{\cB ^{\prime }}\left( \phi \left( X\right) t_{\cB ^{\prime }}\left( \op{a}\right) \right) =\varepsilon _{\cB ^{\prime }}\left( \phi \left( X\right) \phi \left( t_{\cB }\left( \op{a}\right) \right) \right) \\
 & = \varepsilon _{\cB ^{\prime }}\left( \phi \left( Xt_{\cB }\left( \op{a}\right) \right) \right) =\varepsilon _{\cB }\left( Xt_{\cB }\left( \op{a}\right) \right) =\omega _{\cB }\left( X\right) \left( a\right)
\end{align*}
for all $X\in \prim \left( \cB \right) $ and $a\in A$. 
Summing up, we have the following result.

\begin{proposition}\label{prop:PP}
There is a well-defined functor $\mathbb{P}:\Bialgd_{A}\to \ALie{A}$ which assigns to every $A$-bialgebroid $\cB $ its Lie algebra of primitive elements $\prim \left( \cB \right) $ with anchor $\omega _{\cB }:\prim \left( \cB \right) \to  \Derk{A} $ sending $X$ to 
\begin{equation}\label{eq:omegaaction}
\omega _{\cB }\left( X\right) :A\to A, \qquad a\mapsto \omega _{\cB }\left(X\right) \left( a\right) = X\centerdot a.  
\end{equation}
\end{proposition}

Next lemma states a property of the primitive elements of an $A$-bialgebroid that we already observed for $\cB_L$ in \eqref{eq:JL2} and that will be useful to prove the universal property in the forthcoming section.

\begin{lemma}
For $\cB $ an $A$-bialgebroid, every primitive element $X\in \prim\left( \cB \right) $ satisfies
\begin{equation}\label{eq:AmodPrim}
Xt\left( \op{a}\right) -t\left( \op{a}\right) X=t\Big(\op{\varepsilon \big( Xt\left( \op{a}\right) \big)}\Big) \qquad \text{and}\qquad Xs\left( a\right) -s\left( a\right) X = s\Big(\varepsilon \big( Xs\left( a\right) \big)\Big)  
\end{equation}
for all $a\in A$. In particular,
\begin{equation}\label{eq:PRIM}
\big[X,\eta_{\cB}(a \otimes \op{b})\big] = \eta_{\cB}\big(X \centerdot \left(a \otimes \op{b}\right)\big)
\end{equation}
for all $X \in \prim(\cB)$ and all $a,b \in A$.
\end{lemma}

\begin{proof}
In view of the definition of $\cB \tensor{A}\cB $ and of $\cB \tak{A}\cB $ we have that
\begin{align*}
Xt\left( \op{a}\right) \tensor{A}1_{\cB }+t\left( \op{a}\right) \tensor{A}X \stackrel{\eqref{eq:tak}}{=} X\tensor{A}s\left( a\right) +1_{\cB }\tensor{A}Xs\left( a\right) .
\end{align*}
By resorting to the left-hand side identity in \eqref{eq:tak}, this relation can be written equivalently as%
\begin{equation}
Xt\left( \op{a}\right) \tensor{A}1_{\cB }+t\left( \op{a}\right) \tensor{A}X = t\left(\op{a}\right) X\tensor{A}1_{\cB }+1_{\cB }\tensor{A}Xs\left( a\right)
\label{eq:Aprim1}
\end{equation}
or
\begin{equation}
Xt\left( \op{a}\right) \tensor{A}1_{\cB }+1_{\cB }\tensor{A}s\left( a\right) X = X\tensor{A}s\left( a\right) +1_{\cB }\tensor{A}Xs\left( a\right) .
\label{eq:Aprim2}
\end{equation}
By applying $\cB \tensor{A}\varepsilon $ to both sides of \eqref{eq:Aprim1} and by recalling that $\varepsilon(X) = 0$, we get
\begin{equation*}
Xt\left( \op{a}\right) = t\left( \op{a}\right) X+t\Big(\op{\varepsilon \big( Xs\left( a\right) \big)}\Big) \stackrel{\eqref{eq:epsiscambia}}{=} t\left( \op{a}\right) X+t\Big(\op{\varepsilon \big( Xt\left( \op{a}\right) \big)}\Big) .
\end{equation*}%
If we apply instead $\varepsilon \tensor{A}\cB $ to \eqref{eq:Aprim2} then we get
\begin{equation*}
s\Big( \varepsilon \big( Xt\left( \op{a}\right) \big) \Big) +s\left( a\right) X=Xs\left( a\right) ,
\end{equation*}%
which gives the other relation in \eqref{eq:AmodPrim}.
\end{proof}


\subsection{An adjunction between \texorpdfstring{$\ALie{A}$}{A-AnchLie} and \texorpdfstring{$\Bialgd_A$}{A-Bialgd}}\label{ssec:adj}

We show now how the Connes-Moscovici's bialgebroid construction provides a left adjoint to the functor $\bP$ in a way that mimics the well-known ``universal enveloping algebra/space of primitives'' adjunction
\begin{equation}\label{eq:ClassAdj}
\begin{gathered}
\xymatrix{
\bialgk \ar @<+0.5ex> @/^/ [d]^-{\prim(-)} \ar@{}[d]|-{\dashv} \\
\liek. \ar @<+0.5ex> @/^/ [u]^-{U_\K(-)}
}
\end{gathered}
\end{equation}

\begin{remark}
Despite being well-known, it seems that no ``classical'' reference explicitly reports the adjunction \eqref{eq:ClassAdj} in the form we stated it here. Nevertheless, it is straightforward to check that the involved functors are well-defined (they are, in fact, slight adjustments of the functors considered in \cite[page 239]{MilnorMoore}) and that they form an adjoint pair. The unit $L \to \prim(U_\K(L))$ (induced by the canonical map $j_L$ of \eqref{eq:jL}) and the counit $U_\K(\prim(B)) \to B$ (the unique algebra morphism extending the Lie algebra inclusion $\prim(B) \subseteq \cL(B)$) are the obvious natural morphism which are proved to be bijective in \cite[Theorem 5.18]{MilnorMoore}.
\end{remark}

\begin{proposition}\label{prop:counit}
Let $(L,\omega)$ be an $A$-anchored Lie algebra. Given an $A$-bialgebroid $\cB$ and given a Lie algebra morphism $\phi_L:L \to \cL(\cB)$ such that $\phi_L(L)\subseteq \prim(\cB)$ and
\begin{equation}\label{eq:phiomega}
\phi_L(X)\centerdot a = \omega(X)(a)
\end{equation}
for all $a \in A$ and for all $X \in L$, there exists a unique morphism of $A$-bialgebroids $\Phi:\cB_L \to \cB$ such that $\Phi\circ J_L = \phi_L$ and it is explicitly given by \eqref{eq:defPhi}.
\end{proposition}

\begin{proof}
Set $U \coloneqq U_\K(L)$ and let $\phi':U \to \cB$ be the unique $\K$-algebra map extending $\phi_L$. In view of \eqref{eq:PRIM} and \eqref{eq:phiomega}, $\eta_\cB$ and $\phi_L$ satisfy \eqref{eq:Llin}. Thus, by Theorem \ref{thm:UEA}, there exists a unique morphism of $\env{A}$-rings $\Phi : \cB_L \to \cB$ satisfying $\Phi\circ J_L = \phi_L$. Moreover, $\eta_\cB : \env{A} \to \cB$ is always a morphism of $A$-bialgebroids, where $\env{A}$ has the $A$-bialgebroid structure from Example \ref{ex:bialgebroids}\ref{item:enveloping}. Therefore, in view of Remark \ref{rem:PanaiteOystaeyen} and of \cite[Theorem 3.2]{PanaiteOystaeyen}, if we show that $\left(1_A,\phi'\right):\left(\K,U\right) \to \left(A,\cB\right) $ is a morphism of bialgebroids, then we can conclude that $\Phi$ is a morphism of $A$-bialgebroids and finish the proof. Equivalently, we need to check that $\phi': U \to \cB$ satisfies
\begin{gather}
\sum \phi'\left( u_{\left( 1\right) }\right) \tensor{A}\phi'\left( u_{\left( 2\right) }\right) = \sum \phi'\left( u\right)_{\left( 1\right) }\tensor{A}\phi'\left( u\right) _{\left(2\right) }  \label{eq:phiprime} \\
\text{and} \qquad \varepsilon _{\cB }\left( \phi'\left( u\right) \right) = \varepsilon _{U}\left( u\right) 1_{A}  \label{eq:phiprime2}
\end{gather}
for all $u\in U$. 
Since, in view of the PBW theorem, $U$ admits a $\K$-basis of the form
\[\{1_U\}\cup\left\{x_{1}\cdots x_{n}\mid n\geq 1,X_{1},\ldots ,X_{n}\in L\right\},\]
it is enough to check \eqref{eq:phiprime} and \eqref{eq:phiprime2} on the elements of this basis. A direct computation shows that $\varepsilon _{\cB }\left( \phi'\left( 1_U\right) \right) = \varepsilon_\cB(1_\cB) = 1_A = \varepsilon_U(1_U)1_A$ and that
\begin{align*}
\varepsilon _{\cB }\left( \phi'\left( x_{1}\cdots x_{n}\right) \right) & = \varepsilon _{\cB }\left( \phi'\left(x_{1}\cdots x_{n-1}\right) \phi'\left( x_{n}\right) \right) \stackrel{\eqref{item:B5}}{=} \varepsilon _{\cB }\left( \phi'\left(x_{1}\cdots x_{n-1} \right) s\left(\varepsilon _{\cB }\left( \phi'\left( x_{n} \right)\right)\right) \right) \\
 & = \varepsilon _{\cB }\left( \phi'\left(x_{1}\cdots x_{n-1} \right) s\left(\varepsilon _{\cB }\left( \phi_L\left( X_{n}\right)\right)\right) \right) = 0 = \varepsilon _{U}\left( x_{1}\cdots x_{n}\right) 1_{A}
\end{align*}
for all $n \geq 1$. Therefore, relation \eqref{eq:phiprime2} holds. Concerning \eqref{eq:phiprime}, we notice first of all that
\begin{equation*}
\phi'\left( 1_{U}\right) \tensor{A}\phi'\left(1_{U}\right) =1_{\cB }\tensor{A}1_{\cB }= \sum \phi'\left( 1_{U}\right) _{\left( 1\right) }\tensor{A}\phi'\left( 1_{U}\right) _{\left(2\right) },
\end{equation*}
which shows that it is satisfied for $u=1_U$, and we prove by induction on $n\geq 1$ that it also holds for $u=x_{1}\cdots x_{n}$, where $X_{1},\ldots ,X_{n}\in L$. For $n=1$ we have
\begin{align*}
\sum \phi'\left( x_{\left( 1\right) }\right) \tensor{A}\phi'\left( x_{\left( 2\right) }\right) & = \phi'\left( x\right) \tensor{A}\phi'\left( 1_{U}\right) + \phi'\left( 1_{U}\right) \tensor{A}\phi'\left( x\right) \\
& = \phi_L(X)\tensor{A}1_{\cB }+1_{\cB }\tensor{A}\phi_L(X) = \sum \phi'\left( x\right) _{\left( 1\right) }\tensor{A}\phi'\left( x\right) _{\left(2\right) }.
\end{align*}
Assume now that \eqref{eq:phiprime} holds for $n \geq 1$, that is, that we have
\begin{equation}\label{eq:indhy}
\sum \phi'\left( \left( x_{1}\cdots x_{n} \right) _{\left( 1\right) }\right) \tensor{A}\phi'\left( \left( x_{1}\cdots x_{n} \right) _{\left( 2\right) }\right) = \sum \phi'\left( x_{1}\cdots x_{n} \right) _{\left( 1\right) }\tensor{A}\phi'\left( x_{1}\cdots x_{n}\right) _{\left( 2\right) }
\end{equation}
for all $X_1,\ldots,X_n \in L$ and hence, in particular, $\sum \phi'\left( \left( x_{1}\cdots x_{n}\right) _{\left( 1\right) }\right) \tensor{A}\phi'\left( \left( x_{1}\cdots x_{n} \right) _{\left( 2\right) }\right) \in \cB  \tak{A} \cB $. Let us compute
\begin{align*}
& \sum \phi'\left( \left( x_{1}x_{2}\cdots x_{n+1} \right) _{\left(1\right) }\right) \tensor{A}\phi'\left( \left( x_{1}x_{2}\cdots x_{n+1} \right) _{\left( 2\right) }\right) \\
& \stackrel{\phantom{(35)}}{=} \sum \phi'\left( \left( x_{1} \right) _{\left( 1\right) }\left(x_{2}\cdots x_{n+1} \right) _{\left( 1\right) }\right) \tensor{A}\phi'\left( \left( x_{1} \right) _{\left( 2\right) }\left( x_{2}\cdots x_{n+1} \right) _{\left( 2\right) }\right) \\
& \stackrel{\phantom{(35)}}{=} \sum \phi'\left( \left( x_{1} \right) _{\left( 1\right) }\right)\phi'\left( \left( x_{2}\cdots x_{n+1} \right) _{\left( 1\right)}\right) \tensor{A}\phi'\left( \left( x_{1} \right) _{\left(2\right) }\right) \phi'\left( \left( x_{2}\cdots x_{n+1} \right)_{\left( 2\right) }\right) \\
& \stackrel{\eqref{eq:prod}}{=} \left(\sum \phi'\left( \left( x_{1} \right) _{\left( 1\right) }\right) \tensor{A} \phi'\left( \left( x_{1} \right) _{\left(2\right) }\right) \right) \left(\sum \phi'\left( \left( x_{2}\cdots x_{n+1} \right) _{\left( 1\right)}\right) \tensor{A} \phi'\left( \left( x_{2}\cdots x_{n+1} \right)_{\left( 2\right) }\right)\right) \\
 & \stackrel{\eqref{eq:indhy}}{=}\left( \phi_L\left(X_{1}\right)\tensor{A}1_{\cB }+1_{\cB }\tensor{A}\phi_L\left(X_{1}\right)\right) \left( \sum \phi'\left( x_{2}\cdots x_{n+1} \right) _{\left( 1\right) }\tensor{A}\phi'\left( x_{2}\cdots x_{n+1} \right) _{\left( 2\right) }\right) \\
 & \stackrel{\eqref{eq:prod}}{=} \sum \phi'\left( x_{1} \right) _{\left( 1\right) }\phi'\left(x_{2}\cdots x_{n+1} \right) _{\left( 1\right) }\tensor{A}\phi'\left( x_{1} \right) _{\left( 2\right) }\phi'\left( x_{2}\cdots x_{n+1} \right) _{\left( 2\right) } \\
 & \stackrel{\phantom{(35)}}{=} \sum \left( \phi'\left( x_{1} \right) \phi'\left( x_{2}\cdots x_{n+1} \right) \right) _{\left( 1\right) }\tensor{A}\left( \phi'\left( x_{1} \right) \phi'\left( x_{2}\cdots x_{n+1} \right) \right) _{\left( 2\right) } \\
 & \stackrel{\phantom{(35)}}{=} \sum \phi'\left( x_{1}x_{2}\cdots x_{n+1} \right) _{\left( 1\right) }\tensor{A}\phi'\left( x_{1}x_{2}\cdots x_{n+1} \right) _{\left(2\right) },
\end{align*}
so that it holds for $n+1$ and we may conclude that it holds for every $n$ by induction. 
\end{proof}

\begin{corollary}\label{cor:unit}
Let $(L,\omega)$ be an $A$-anchored Lie algebra. Given an $A$-bialgebroid $\cB$ and given a morphism of $A$-anchored Lie algebras $\phi_L:(L,\omega) \to \left(\prim(\cB),\omega_\cB\right)$, there exists a unique morphism of $A$-bialgebroids $\Phi:\cB_L \to \cB$ such that $\Phi\circ J_L = \theta_\cB \circ \phi_L$, where $\theta_\cB:\prim(\cB) \to \cB$ is the canonical inclusion.
\end{corollary}

\begin{proof}
The morphism of Lie algebras $\left(\theta_\cB \circ \phi_L\right):L \to \cL(\cB)$ satisfies \eqref{eq:phiomega} if and only if $\phi_L$ is of $A$-anchored Lie algebras.
\end{proof} 

\begin{theorem}\label{th:main}
The assignment 
\[
\bB : \ALie{A}\to \Bialgd_{A}, \qquad \left( L,\omega \right) \mapsto A\odot U_{\K}\left( L\right) \odot A
\] 
induces a well-defined functor which is left adjoint to the functor
\[
\bP:\Bialgd_{A}\to \ALie{A}, \qquad \cB \mapsto \left(\prim \left( \cB \right), \omega_\cB \right),
\] 
where $\omega _{\cB }:\prim \left( \cB \right) \to \Derk{A} $ is the anchor of equation \eqref{eq:omegaaction}. Write $\vartheta _{\cB }:U_{\K}\left( \prim \left( \cB \right) \right) \to \cB $ for the unique algebra map extending the inclusion $\theta_\cB:\prim \left( \cB \right) \subseteq \cB $. Then the unit and the counit of this adjunction are given by
\begin{gather}
\gamma _{L} :\left( L,\omega \right) \to \left( \prim \left( A\odot U_{\K}\left( L\right) \odot A\right) , \omega _{A\odot U_{\K}\left( L\right) \odot A}\right) , \qquad X\mapsto J_L(X) = 1_{A}\otimes x\otimes 1_{A}, \label{eq:unit} \\
\text{and} \quad \epsilon _{\cB } : A\odot U_{\K}\left( \prim \left( \cB \right) \right)\odot A\to \cB , \qquad a\otimes u\otimes b\mapsto s_{\cB }\left( a\right)t_{\cB }\left( \op{b}\right) \vartheta _{\cB }\left( u\right), \notag
\end{gather}
respectively, where $x = j_L(X)$, as usual. Furthermore, every component of the unit is a monomorphism and hence $\bB $ is faithful.
\end{theorem}

\begin{proof}
We need to see how $\bB$ operates on morphisms. Let $f:(L,\omega) \to (L',\omega')$ be a morphism of $A$-anchored Lie algebras. In view of the fact that
\begin{align*}
\omega _{\cB _{L}}\left( \gamma _{L}\left( X\right) \right) \left( a\right) & \stackrel{\eqref{eq:unit}}{=} \omega _{\cB _{L}}\left( 1_{A}\otimes x\otimes 1_{A}\right) \left( a\right) \stackrel{\eqref{eq:omegaaction}}{=} \left( 1_{A}\otimes x\otimes 1_{A}\right) \centerdot a \\
 & \stackrel{\eqref{eq:dotaction}}{=} \varepsilon _{\cB _{L}}\left( \left( 1_{A}\otimes x\otimes 1_{A}\right)\left( a\otimes 1_{U}\otimes 1_A\right) \right)  \\
 & \stackrel{\eqref{eq:CMprod}}{=} \varepsilon _{\cB _{L}}\left( \left( x\cdot a\right) \otimes 1_{U}\otimes 1_A+a\otimes x\otimes 1_A+a\otimes 1_{U}\otimes \left(x\cdot 1_A\right) \right)  \\
 & \stackrel{\eqref{eq:Uaction}}{=} x\cdot a + a\varepsilon _{U}\left( x\right) \stackrel{\eqref{eq:Uaction}}{=} \omega \left(X\right) \left( a\right)
\end{align*}
for all $X \in L$ and $a \in A$, the morphism $\gamma_L:L \to \prim(\cB_L)$ induced by $J_L$ (namely, we have $J_L = \theta_{\cB_L} \circ \gamma_L$) is a morphism of $A$-anchored Lie algebras for every $L$ in $\ALie{A}$ and hence, by Corollary \ref{cor:unit}, there exists a unique morphism of $A$-bialgebroids $\Phi:\cB_L \to \cB_{L'}$ such that $\Phi \circ J_L = J_{L'} \circ f = \theta_{\cB_{L'}} \circ \gamma_{L'} \circ f$ and it is explicitly given by \eqref{eq:defPhi}. We set $\bB(f) = \Phi$.

In order to conclude, consider the natural assignment
\begin{equation}\label{eq:adjunction}
\Bialgd_{A}\left(A\odot  U_\K(L) \odot  A,\cB\right) \to \ALie{A}\left(L,\prim(\cB)\right), \qquad \Psi \mapsto \prim(\Psi) \circ \gamma_L.
\end{equation}
Corollary \ref{cor:unit} states that for every $\phi_L$ in $\ALie{A}\left(L,\prim(\cB)\right)$ there exists a unique $\Phi$ in $\Bialgd_{A}\left(A\odot  U_\K(L) \odot  A,\cB\right)$ such that $\Phi \circ J_L = \theta_\cB \circ \phi_L$, which implies that \eqref{eq:adjunction} is bijective.
\end{proof}

\begin{remark}
In the context of the proof above, let $F : U \to U'$ be the unique $\K$-algebra morphism satisfying $F \circ j_L = j_{L'} \circ f$. Since
\[\psi:U \to \cB_{L'}, \qquad u \mapsto 1_A \otimes F(u) \otimes 1_A,\]
is a $\K$-algebra morphism such that $\psi\circ j_L = J_{L'}\circ f$, the unique morphism of $A$-bialgebroids $\bB(f):\cB_L \to \cB_{L'}$ induced by $f : L \to L'$ has the form
\[
\bB(f) \left( a\otimes u\otimes b\right) = a \otimes F(u) \otimes b
\]
for all $a,b \in A$ and $u \in U$. 
Thus, if $f:(L,\omega) \to (L',\omega')$ is a morphism of $A$-anchored Lie algebras, we will also write $A\odot U(f)\odot A$ to denote the $A$-bialgebroid morphism $\bB(f)$, where $U(f):U_\K(L) \to U_\K(L')$ is the unique $\K$-algebra morphism such that $U(f) \circ j_L = j_{L'} \circ f$. 
\end{remark}

\begin{example}\label{ex:Ders}
Let $A$ be a finite-dimensional algebra and let $\cB = \End{\K}{A}$ with the $A$-bialgebroid structure of Example \ref{ex:bialgebroids}\ref{item:exb}. It is clear that $\prim(\cB) = \Derk{A}$. Then, the image of the associated Connes-Moscovici's bialgebroid $A \odot U_\K(\Derk{A}) \odot A$ in $\End{\K}{A}$ via $\epsilon_\cB$ is the $\env{A}$-subring of $\End{\K}{A}$ generated by $\env{A}$ and $\Derk{A}$. In this sense, it can be interpreted as the \emph{derivation $\env{A}$-ring} of $A$ in the sense of \cite[Chapter 15, \S1.4]{McConnelRobson}.
\end{example}

In particular, the counit of the adjunction is not surjective in general. We will see with Corollary \ref{cor:primitive} why also the unit is not surjective.


\section{An intrinsic description of \texorpdfstring{$A \odot U_\K(L) \odot A$}{A x U(L) x A}}\label{sec:intrinsic}

Inspired by the results of Milnor-Moore and Moerdijk-Mr\v{c}un, which give an intrinsic description of those bialgebras/bialgebroids that are universal enveloping algebras of Lie algebras/Lie-Rinehart algebras, we look for necessary and sufficient conditions on an $A$-bialgebroid $\cB$ in order to claim that it is a Connes-Moscovici's bialgebroid $A \odot U_\K(L) \odot A$ for some $A$-anchored Lie algebra $(L,\omega)$.


\subsection{The primitives of the Connes-Moscovici's bialgebroid}\label{ssec:CMprimitives}

To begin with, we need a more detailed analysis of the space of primitives of $\cB_L$.

\begin{lemma}\label{lem:LieIdeal}
Let $\cB$ be an $A$-bialgebroid. Then the $\K$-vector subspace
\[\langle s-t \rangle \coloneqq \left\{s(a) - t(\op{a}) \mid a \in A\right\} \subseteq \cB\]
is an ideal in $\prim(\cB)$. The Lie bracket is explicitly given by
\[
\big[s(a) - t(\op{a}),s(b) - t(\op{b})\big] = s\big([a,b]\big) - t\big(\op{[a,b]}\big)
\]
and the anchor by
\[
\omega_{\cB}':\langle s-t \rangle \to \Derk{A}, \qquad s(a) - t(\op{a}) \mapsto [a,-],
\]
for all $a,b \in A$.
\end{lemma}

\begin{proof}
The fact that $\langle s-t \rangle$ is contained in $\prim(\cB)$ follows from 
\[\Delta_\cB(s(a) - t(\op{a})) \stackrel{\eqref{eq:DeltaLin}}{=} s(a) \tensor{A} 1_{\cB} - 1_\cB \tensor{A} t(\op{a}) \stackrel{\eqref{eq:tak}}{=} \big(s(a) - t(\op{a})\big) \tensor{A} 1_\cB - 1_\cB \tensor{A} \big(s(a) - t(\op{a})\big). \]
The fact that it is a Lie ideal in $\prim(\cB)$ with respect to the commutator bracket follows because
\begin{align*}
\big[X, s(a) - t(\op{a})\big] & = Xs(a) - Xt(\op{a}) - s(a)X + t(\op{a})X \stackrel{\eqref{eq:AmodPrim}}{=} s\big(\varepsilon\left(Xs(a)\right)\big) - t\big(\op{\varepsilon\left(Xt(\op{a})\right)}\big) \\
 & \stackrel{\eqref{item:B5}}{=} s\big(\varepsilon\left(Xs(a)\right)\big) - t\big(\op{\varepsilon\left(Xs(a)\right)}\big) \in \langle s-t \rangle.
\end{align*}
Now, a direct computation shows that
\[\big[s(a) - t(\op{a}),s(b) - t(\op{b})\big] = s\big([a,b]\big) - t\big(\op{[a,b]}\big) \]
for all $a,b \in A$ as claimed. Furthermore,
\[
\omega_\cB\big(s(a)-t(\op{a})\big)(b) \stackrel{\eqref{eq:omegaaction}}{=} \big(s(a)-t(\op{a})\big)\centerdot b \stackrel{\eqref{eq:dotaction}}{=} \varepsilon_\cB(s(a)s(b))-\varepsilon_\cB(t(\op{a})t(\op{b})) = [a,b]
\]
for all $a,b \in A$ and hence the proof is concluded.
\end{proof}

Lemma \ref{lem:LieIdeal} makes it clear why the counit $\epsilon_\cB : \cB_{\prim(\cB)} \to \cB$ of Theorem \ref{th:main} is not injective in general. Namely, every element of the form $a \otimes 1_U \otimes 1_A - 1_A \otimes 1_U \otimes a - 1_A \otimes \left(s(a) - t(\op{a})\right) \otimes 1_A$ for $a \in A$ lives in the kernel of $\epsilon_\cB$.

\begin{example}
Consider the $A$-bialgebroid $\cB = \End{\K}{A}$ from Example \ref{ex:bialgebroids}\ref{item:exb}. As we have seen in Example \ref{ex:Ders}, the primitive elements of $\cB$ are the derivations of $A$. It is easy to see that $\langle s-t\rangle$ are exactly the inner derivations. 
\end{example}

For $H$ a Hopf algebra and $A$ an $H$-module algebra, we will often write $1_{A}\otimes \prim(H) \otimes 1_{A}$ for $\left\{1_A \otimes h \otimes 1_A \mid h \in \prim(H)\right\}$.

\begin{proposition}\label{prop:primitive}
For any Hopf algebra $H$ and any $H$-module algebra $A$, we have that $1_A \otimes \prim(H) \otimes 1_A$ and $\langle s-t\rangle$ are subalgebras of $ \prim \left( A\odot  H \odot A\right)$. Moreover, we have an isomorphism of $A$-anchored Lie algebras
\[
\big(1_A \otimes \prim(H) \otimes 1_A\big) \ltimes_{\delta} \langle s-t\rangle \cong \prim \left( A\odot  H \odot A\right)
\]
given by the vector space sum.
\end{proposition}

\begin{proof}
The fact that $\langle s-t\rangle$ is an ideal in $\prim(A \odot  H \odot  A)$ has been established in Lemma \ref{lem:LieIdeal} for a general $A$-bialgebroid. The fact that $1_A \otimes \prim(H) \otimes 1_A$ is a subalgebra is a straightforward computation.
Moreover, notice that if $1_A \otimes X \otimes 1_A = a \otimes 1_H \otimes 1_A - 1_A \otimes 1_H \otimes a$ for some $X \in \prim(H)$ and some $a \in A$, then
\begin{align*}
0 & = (A \otimes \varepsilon_H \otimes A)(1_A \otimes X \otimes 1_A) = (A \otimes \varepsilon_H \otimes A)(a \otimes 1_H \otimes 1_A - 1_A \otimes 1_H \otimes a) \\
 & = a \otimes 1_A - 1_A \otimes a
\end{align*}
which implies that
\[a \otimes 1_H \otimes 1_A - 1_A \otimes 1_H \otimes a = \eta(a \otimes \op{1_A} - 1_A \otimes \op{a})(1_A \otimes 1_H \otimes 1_A) = 0.\]
In view of Lemma \ref{lem:sdprod}, we are left to check that
\begin{equation}\label{eq:incl}
\prim \left( A\odot  H \odot A\right) \subseteq (1_A \otimes \prim(H) \otimes 1_A) + \langle s-t\rangle.
\end{equation}

Let us consider a primitive element $\xi \in A\odot  H \odot A $. Fix a basis $\left\{e_{i}\mid i\in S\right\} $ for $A$ as a vector space, where $S$ is some set of indexes with a distinguished index $0$ and $e_{0}=1_{A}$. Write
\begin{equation*}
\xi = \sum_{i,j}e_{i}\otimes h_{ij}\otimes e_{j}
\end{equation*}
where almost all the $h_{ij}$ are $0$. Consider also the dual elements $\left\{ e_{i}^{\ast }\mid i\in S\right\} $ of the $e_{i}$'s. Since $\xi$ is primitive, the following relation holds
\begin{equation}\label{eq:primitive1}
\begin{gathered}
\sum_{i,j}\left( e_{i}\otimes \left( h_{ij}\right) _{\left( 1\right) }\otimes 1_{A}\right) \tensor{A}\left( 1_{A}\otimes \left( h_{ij}\right) _{\left( 2\right) }\otimes e_{j}\right) \\
 =\sum_{i,j}\left( e_{i}\otimes h_{ij}\otimes e_{j}\right) \tensor{A}\left( 1_{A}\otimes 1_{H}\otimes 1_{A}\right) +\left( 1_{A}\otimes 1_{H}\otimes 1_{A}\right) \tensor{A}\left( e_{i}\otimes h_{ij}\otimes e_{j}\right) .  
\end{gathered}
\end{equation}
For $k \neq 0 \neq l$, let us apply the $\K$-linear morphism $\left( e_{k}^{\ast }\otimes H\otimes A\right) \tensor{A}\left( A\otimes H\otimes e_{l}^{\ast }\right) $ to both sides of the identity \eqref{eq:primitive1}. We find out that 
\begin{equation*}
\sum \left( \left( h_{kl}\right) _{\left( 1\right) }\otimes 1_{A}\right) \tensor{A}\left( 1_{A}\otimes \left( h_{kl}\right) _{\left( 2\right) }\right) =0
\end{equation*}
from which it follows that $h_{kl}=0$, by applying $\left( H\otimes A\right) \tensor{A}\left( A\otimes \varepsilon _{H}\right) $ to both sides again. Therefore,
\begin{equation*}
\xi = 1_{A}\otimes h_{00}\otimes 1_{A}+\sum_{i\neq 0}e_{i}\otimes h_{i0}\otimes 1_{A}+\sum_{i\neq 0}1_{A}\otimes h_{0i}\otimes e_{i}.
\end{equation*}
Consider again the identity \eqref{eq:primitive1}, that now rewrites
\begin{gather*}
\left[ 
\begin{array}{c}
\sum \left( 1_{A}\otimes \left( h_{00}\right) _{\left( 1\right) }\otimes 1_{A}\right) \tensor{A}\left( 1_{A}\otimes \left( h_{00}\right) _{\left( 2\right) }\otimes 1_{A}\right) + \\ 
+\sum_{i\neq 0}\left( e_{i}\otimes \left( h_{i0}\right) _{\left( 1\right) }\otimes 1_{A}\right) \tensor{A}\left( 1_{A}\otimes \left( h_{i0}\right) _{\left( 2\right) }\otimes 1_{A}\right) + \\ 
+\sum_{i\neq 0}\left( 1_{A}\otimes \left( h_{0i}\right) _{\left( 1\right) }\otimes 1_{A}\right) \tensor{A}\left( 1_{A}\otimes \left( h_{0i}\right) _{\left( 2\right) }\otimes e_{i}\right)
\end{array}
\right] \\
 = \left[ 
\begin{array}{c}
\left( 1_{A}\otimes h_{00}\otimes 1_{A}\right) \tensor{A}\left( 1_{A}\otimes 1_{H}\otimes 1_{A}\right) +\left( 1_{A}\otimes 1_{H}\otimes 1_{A}\right) \tensor{A}\left( 1_{A}\otimes h_{00}\otimes 1_{A}\right) + \\ 
+\sum_{i\neq 0}\left( e_{i}\otimes h_{i0}\otimes 1_{A}\right) \tensor{A}\left( 1_{A}\otimes 1_{H}\otimes 1_{A}\right) +\left( 1_{A}\otimes 1_{H}\otimes 1_{A}\right) \tensor{A}\left( e_{i}\otimes h_{i0}\otimes 1_{A}\right) + \\ 
+\sum_{i\neq 0}\left( 1_{A}\otimes h_{0i}\otimes e_{i}\right) \tensor{A}\left( 1_{A}\otimes 1_{H}\otimes 1_{A}\right) +\left( 1_{A}\otimes 1_{H}\otimes 1_{A}\right) \tensor{A}\left( 1_{A}\otimes h_{0i}\otimes e_{i}\right)
\end{array}
\right] .
\end{gather*}
If we apply $\left( 1_{A}^{\ast }\otimes H\otimes A\right) \tensor{A} \left( A\otimes H\otimes 1_{A}^{\ast }\right) $ then we get that
\begin{equation}\label{eq:primitive3}
\begin{gathered}
\sum \left( \left( h_{00}\right) _{\left( 1\right) }\otimes 1_{A}\right) \tensor{A}\left( 1_{A}\otimes \left( h_{00}\right) _{\left( 2\right) }\right) = \\ 
\left[
\begin{array}{c} 
\left( h_{00}\otimes 1_{A}\right) \tensor{A}\left( 1_{A}\otimes 1_{H}\right) +\left( 1_{H}\otimes 1_{A}\right) \tensor{A}\left(1_{A}\otimes h_{00}\right) + \\
+ \sum_{i\neq 0}\left( 1_{H}\otimes 1_{A}\right) \tensor{A}\left( e_{i}\otimes h_{i0}\right) +\left( h_{0i}\otimes e_{i}\right) \tensor{A}\left( 1_{A}\otimes 1_{H}\right) 
\end{array}\right].
\end{gathered}
\end{equation}
By resorting to the $\K$-linear isomorphism $\left( H\otimes A\right) \tensor{A}\left( A\otimes H\right) \cong H\otimes A\otimes H$, the equality \eqref{eq:primitive3} becomes 
\begin{equation} \label{eq:primitives2}
\sum \left( h_{00}\right) _{\left( 1\right) }\otimes 1_{A}\otimes \left( h_{00}\right) _{\left( 2\right) } = h_{00}\otimes 1_{A}\otimes 1_{H}+1_{H}\otimes 1_{A}\otimes
h_{00}+\sum_{i\neq 0}h_{0i}\otimes e_{i}\otimes 1_{H}+1_{H}\otimes e_{i}\otimes h_{i0}.
\end{equation}
By applying $H\otimes 1_{A}^{\ast }\otimes H$ to both sides of the identity \eqref{eq:primitives2} we get that
\begin{equation*}
\sum \left( h_{00}\right) _{\left( 1\right) }\otimes \left( h_{00}\right) _{\left( 2\right) }=h_{00}\otimes 1_{H}+1_{H}\otimes h_{00},
\end{equation*}
whence $h_{00}$ is primitive in $H$, and by applying $H\otimes e_{k}^{\ast}\otimes H$ for all $k\neq 0$ we get 
\begin{equation}
1_{H}\otimes h_{k0}+h_{0k}\otimes 1_{H}=0.  \label{eq:primitive4}
\end{equation}
By applying further $H \otimes \varepsilon_H$ we find that $h_{0k} = -\varepsilon(h_{k0})1_H$ and hence from \eqref{eq:primitive4} we deduce that
\[ 0 = 1_{H}\otimes h_{k0}-\varepsilon(h_{k0})1_H\otimes 1_{H} = 1_H \otimes \left(h_{k0} - \varepsilon(h_{k0})1_H\right), \]
which in turn entails that $h_{k0} - \varepsilon(h_{k0})1_H = 0$ by applying $\varepsilon_H \otimes H$ to both sides, that is, $h_{k0} = \varepsilon(h_{k0})1_H$ for all $k\neq 0$. Summing up,
\begin{equation*}
\xi = 1_{A}\otimes h_{00}\otimes 1_{A} + \sum_{k\neq 0}\varepsilon(h_{k0})\left(e_{k}\otimes 1_H\otimes 1_{A} - 1_{A}\otimes 1_H\otimes e_{k}\right)
\end{equation*}
which proves that the inclusion \eqref{eq:incl} holds.
\end{proof}

\begin{corollary}\label{cor:primitive}
For any $A$-anchored Lie algebra $(L,\omega)$ we have 
\[\prim \left( A\odot U_{\K}\left(L\right) \odot A\right) \cong L \ltimes_{\delta} \langle s-t\rangle.\]
\end{corollary}

\begin{proof}
It follows from \cite[Theorem 5.18]{MilnorMoore} that $\prim(U_\K(L)) = L$. Moreover, it is clear that $L \cong 1_A \otimes L \otimes 1_A$ as $A$-anchored Lie algebras.
\end{proof}

It is evident from Corollary \ref{cor:primitive} why, in general, the unit $\gamma_{L}: L \to \prim(\cB_L)$ from Theorem \ref{th:main} cannot be surjective.


\subsection{Primitively generated bialgebroids}\label{ssec:primgen}

Let $\cB $ be an $A$-bialgebroid and consider the $\env{A}$-bimodule 
\[
M_\cB  \coloneqq \env{A} \otimes \prim(\cB ) \otimes \env{A}.
\]
There is a canonical $\env{A}$-bilinear map $\varphi_\cB  : M_\cB  \to \cB $ given by
\[
\xymatrix{
\env{A} \otimes \prim(\cB ) \otimes \env{A} \ar[rr]^-{\eta_\cB  \otimes \theta_\cB  \otimes \eta_\cB } && \cB  \otimes \cB  \otimes \cB  \ar[r]^-{m_2} & \cB ,
}
\]
where $\theta_\cB  : \prim(\cB ) \to \cB $ is the inclusion and $m_2:\cB  \otimes \cB  \otimes \cB  \to \cB , x \otimes y \otimes z \mapsto xyz$. Therefore, by the universal property of the tensor $\env{A}$-ring $T_{\env{A}}\left(M_\cB \right)$ (considered as the free $\env{A}$-ring on the $\env{A}$-bimodule $M_\cB $; see \cite[Theorem VII.3.2]{MacLane} or \cite[Proposition 1.4.1]{Nichols}), there exists a unique morphism of $\env{A}$-rings
\begin{equation}\label{eq:Phi}
\Phi_\cB :T_{\env{A}}\left(M_\cB \right) \to \cB 
\end{equation}
that extends $\varphi_\cB $. We set $\varsigma_\cB:M_\cB \to T_{\env{A}}\left(M_\cB \right)$ for the canonical inclusion. 

Since $T_{\env{A}}\left(M_\cB \right)$ is a graded $\env{A}$-ring with grading given by
\[T_{\env{A}}\left(M_\cB \right)_0 \coloneqq \env{A} \qquad \text{and} \qquad T_{\env{A}}\left(M_\cB \right)_n \coloneqq M_\cB  \tensor{\env{A}} \cdots \tensor{\env{A}} M_\cB  \qquad (n \text{ times}),\]
for all $n\geq 1$, $\cB $ inherits a natural filtration as $\env{A}$-ring given by
\begin{equation}\label{eq:filtration}
F_n(\cB ) \coloneqq \Phi_\cB \left(\bigoplus_{k=0}^{n}T_{\env{A}}\left(M_\cB \right)_k\right).
\end{equation}
We set by definition $F_{-1}(\cB ) \coloneqq 0$, as usual. Notice that $F_1(\cB )$ is the $\env{A}$-subbimodule of $\cB $ generated by $\env{A}$ and $\prim(\cB )$, thus we call $\{F_n(\cB) \mid n \in \N\}$ the \emph{primitive filtration}.

\begin{example}\label{ex:CMfiltration}
Let $(L,\omega)$ be an $A$-anchored Lie algebra and consider the Connes-Moscovici's bialgebroid $\cB _L = A \odot  U_\K(L) \odot  A$. Recall that $U_\K(L)$ is a filtered $\K$-algebra with filtration induced by the canonical projection $T_\K(L) \to U_\K(L)$. Then 
\begin{equation}\label{eq:CMfiltration}
F_n\left(A \odot  U_\K(L) \odot  A\right) = A \otimes F_n\left(U_\K(L)\right) \otimes A.
\end{equation}
\end{example}

\begin{definition}\label{def:primgen}
We say that an $A$-bialgebroid $\cB $ is \emph{primitively generated} if $\cB  = \bigcup_{n\geq 0}F_n(\cB )$.
\end{definition}

\begin{remark}\label{rem:primgen}
Definition \ref{def:primgen} is given in the same spirit of \cite[page 239]{MilnorMoore}. 
In particular, the Connes-Moscovici's bialgebroid $A\odot  U_\K(L)\odot  A$ of an $A$-anchored Lie algebra $(L,\omega)$ is primitively generated. Notice also that $\cB $ is primitively generated if and only if $\Phi_\cB$ of \eqref{eq:Phi} is surjective.
\end{remark}

\begin{proposition}\label{lem:PhiNat}
The canonical morphism $\Phi_\cB$ is natural in $\cB \in \Bialgd_A$. Namely, every morphism $\phi:\cB \to \cB'$ of $A$-bialgebroids induces a morphism $T_\phi:T_{\env{A}}\left(M_\cB \right) \to T_{\env{A}}\left(M_{\cB'} \right)$ of graded $\env{A}$-rings in a functorial way and the following diagram commutes
\begin{equation}\label{eq:natural}
\begin{gathered}
\xymatrix{
T_{\env{A}}\left(M_{\cB}\right) \ar[r]^-{T_{\phi}} \ar[d]_-{\Phi_{\cB}} & T_{\env{A}}\left(M_{\cB'}\right) \ar[d]^-{\Phi_{\cB'}} \\
\cB_{L} \ar[r]_-{\phi} & \cB.
}
\end{gathered}
\end{equation}
\end{proposition}

\begin{proof}
By the universal property of the tensor $\env{A}$-ring and a standard argument, any morphism $\phi:\cB \to \cB'$ of $A$-bialgebroids induces a unique morphism $T_\phi : T_{\env{A}}\left(M_\cB \right) \to T_{\env{A}}\left(M_{\cB'} \right)$ of graded $\env{A}$-rings extending $\env{A} \otimes \bP(\phi) \otimes \env{A}$. It satisfies $\phi \circ \Phi_\cB = \Phi_{\cB'}\circ T_\phi$.
\end{proof}

\begin{corollary}\label{cor:surjectivity}
Any morphism $\phi:\cB \to \cB'$ of $A$-bialgebroids is filtered with respect to the primitive filtration, that is to say, $\phi \left(F_n\left(\cB\right)\right) \subseteq F_n(\cB' )$ for all $n \geq 0$. In particular, the component of the counit $\epsilon_\cB $ corresponding to an $A$-bialgebroid $\cB $ is a filtered morphism.
Furthermore, $\epsilon_\cB $ is surjective if and only if $\cB $ is primitively generated.
\end{corollary}

\begin{proof}
The fact that any morphism of $A$-bialgebroids is filtered follows from the commutativity of \eqref{eq:natural} and the definition of the primitive filtration. Now, set $U \coloneqq U_\K(\prim(\cB ))$ and $L \coloneqq \prim(\cB)$. Since, by Proposition \ref{prop:primitive},
\[\prim\left(A \odot  U \odot  A\right) = \left(1_A \otimes L \otimes 1_A \right) \oplus \left\langle s_{\cB_{L}} - t_{\cB_{L}} \right\rangle\]
and since $\bP(\epsilon_\cB)\left(1_A \otimes X \otimes 1_A\right) = X$ for all $X \in \prim(\cB)$, it is clear that $\bP(\epsilon_\cB)$ is surjective. Thus, $T_{\epsilon_{\cB}}: T_{\env{A}}\left(M_{\cB_{L}}\right) \to T_{\env{A}}\left(M_\cB\right)$ is surjective as well and we know from Example \ref{ex:CMfiltration} that $\Phi_{\cB_{L}}$ is surjective. Since, by naturality of $\Phi_\cB$, the following diagram commutes
\[
\xymatrix{
T_{\env{A}}\left(M_{\cB_{L}}\right) \ar@{->>}[r]^-{T_{\epsilon_\cB}} \ar@{->>}[d]_-{\Phi_{\cB_{L}}} & T_{\env{A}}\left(M_\cB\right) \ar[d]^-{\Phi_\cB} \\
\cB_{L} \ar[r]_-{\epsilon_\cB} & \cB,
}
\]
$\epsilon_\cB $ is surjective if and only if $\Phi_\cB$ is.
\end{proof}

\begin{example}\label{ex:matrices}
Let $A = \mathsf{Mat}_n(\K)$ be the algebra of $n \times n$ matrices with coefficients in $\K$, for $n\geq 2$.
Consider again the bialgebroid $\cB = \End{\K}{A}$ of Example \ref{ex:bialgebroids}\ref{item:exb}, its space of primitive elements $\prim(\cB) = \Derk{A}$ and the associated Connes-Moscovici's bialgebroid $A \odot U_\K(\Derk{A}) \odot A$ as in Example \ref{ex:Ders}. 
Every $f \in \End{\K}{A}$ is uniquely determined by the images $f\left(E_{i,j}\right) = \sum_{h,k}f_{h,k}^{i,j}E_{h,k}$ for all $i,j$ and it satisfies
\[f(M) = \sum_{h,k}\left(\sum_{i,j}m_{i,j}f_{h,k}^{i,j}\right)E_{h,k} = \sum_{h,k}\sum_{i,j} f_{h,k}^{i,j}E_{h,i}ME_{j,k} = \eta_\cB\left(\sum_{i,j,h,k}f_{h,k}^{i,j}E_{h,i}\ \otimes E_{j,k}\right)(M),\]
whence $\epsilon_\cB$ is surjective and $\cB$ is primitively generated.
\end{example}

Let $\cB $ be an $A$-bialgebroid. 
Recall that, given the filtered left $\env{A}$-module $\cB $ with filtration $\left\{F_n(\cB )\mid n\geq 0 \right\}$ as in \eqref{eq:filtration}, we can consider the associated graded left $\env{A}$-module $\gr(\cB )$ as in \S\ref{ssec:filtgrad}.
Being $\epsilon_\cB $ filtered, it induces a left $\env{A}$-linear homomorphism
\[\gr(\epsilon_\cB ) : \gr\left(A \odot  U(\prim(\cB )) \odot  A\right) \to \gr(\cB ).\]

The following lemma, which should be well-known, is implicitly needed in the proof of Theorem \ref{thm:HR} below. Its statement resembles closely \cite[Remark 2.4]{MoerdijkLie}. Its proof can be deduced from the results in \cite[Appendix B]{LaiachiPaolo-complete} and it follows closely the argument reported in \cite[page 229]{Sweedler} for coalgebras over a field, but we sketch it here for the sake of the reader.

\begin{lemma}\label{lem:supertech}
Let $\cB $ be a primitively generated $A$-bialgebroid. Then $\left(\sMto{\cB},\Delta_\cB,\varepsilon_\cB\right)$ is a filtered $A$-coring with the primitive filtration. In particular, if $\gr_n(\cB )$ is a projective left $\env{A}$-module for all $n \geq 0$, then $F_n(\cB )$ is a projective left $\env{A}$-module for all $n\geq 0$, the map
\[\gr(\cB ) \tensor{A} \gr(\cB ) \to \gr(\cB  \tensor{A} \cB ), \quad (x + F_h(\cB )) \tensor{A} (y + F_{k}(\cB )) \mapsto x\tensor{A}y + F_{h+k+1}(\cB  \tensor{A} \cB ),\]
is an isomorphism of left $\env{A}$-modules and the $A$-coring structure on $\cB $ induces an $A$-coring structure on $\gr(\cB )$. Furthermore, any morphism $\phi :\cB \to \cB' $ of primitively generated $A$-bialgebroids which are graded projective as $A$-corings induces a morphism of graded $A$-corings $\gr(\phi ):\gr\left(\cB\right) \to \gr(\cB' )$.
\end{lemma}

\begin{proof}
It follows from Lemma \ref{lemma:locfgp} and Proposition \ref{prop:GrAss}, once proved that $\cB$ with the filtration $\{F_n(\cB)\mid n \in \N\}$ is a filtered $A$-coring. To this aim observe that, as a left $\env{A}$-submodule of $\cB$, $F_n(\cB )$ is generated by $1_\cB$ and by elements of the form
$
X_{1}\cdots X_{k}
$
for $1\leq k \leq n$ and $X_{i} \in \prim(\cB )$ for all $i$. By applying $\Delta_\cB$ we find that
\begin{gather*}
\Delta_\cB(1_\cB) = 1_\cB \tensor{A} 1_\cB \in F_0(\cB) \tensor{A} F_0(\cB) \qquad \text{and that} \\
\Delta_\cB\left(X_{1}\cdots X_{k}\right) = \prod_{i=1}^k\left(X_i \tensor{A} 1_\cB + 1_\cB \tensor{A} X_i\right) = \sum_{t+s=k}X_{p_1}\cdots X_{p_t} \tensor{A} {X_{q_1}}\cdots {X_{q_s}}
\end{gather*}
belongs to $\sum_{t+s = k}F_t(\cB ) \tensor{A} F_s(\cB )$, where $\{p_i\mid 1\leq i\leq t\}\cup\{q_j\mid 1\leq j\leq s\} = \{1,\ldots,k\}$, $p_1<p_2<\cdots<p_t$ and $q_1<q_2<\cdots<q_s$. By left $\env{A}$-linearity of $\Delta_\cB$, we may conclude that it is filtered. On the other hand, $\varepsilon_\cB$ is obviously filtered (by definition of the filtration on $A$). Therefore, $(\sMto{\cB},\Delta_\cB,\varepsilon_\cB)$ is in fact a filtered $A$-coring.
\end{proof}

In view of Lemma \ref{lem:supertech} and by mimicking \cite[page 3140]{MoerdijkLie} and \S\ref{ssec:Acorings}, we give the following definition.

\begin{definition}
A primitively generated $A$-bialgebroid $\cB$ is called \emph{graded projective} if each associated graded component $\gr_n(\cB)$ is a projective left $\env{A}$-module.
\end{definition}

The following can be understood as an analogue of the celebrated Heyneman-Radford Theorem for coalgebras \cite[Proposition 2.4.2]{HeynemanRadford} (extending the earlier Heyneman-Sweedler Theorem \cite[Lemma 3.2.6]{HeynemanSweedler}).

\begin{theorem}\label{thm:HR}
Let $\phi:\cB \to \cB'$ be a morphism of graded projective primitively generated $A$-bialgebroids. If $\gr\left(\cB\right)$ is strongly graded as an $A$-coring and $\phi$ is injective when restricted to the left $\env{A}$-submodule of $\cB$ generated by $\prim\left(\cB\right)$, then $\phi$ is injective.
\end{theorem}

\begin{proof}
In order to prove that $\phi$ is injective, we are going to prove that $\gr\left(\phi\right)$ is injective. In view of \cite[Chapter D, Corollary III.6]{NastasescuOystaeyen}, the latter implies that $\phi$ is injective as well.

To this aim, let us prove that $\gr_n\left(\phi\right)$ is injective for every $n \geq 0$. To begin with, let us prove that $\gr_0(\phi)$ is injective. Since $\gr_0\left(\cB\right) = F_0\left(\cB\right)/F_{-1}\left(\cB\right) = F_0\left(\cB\right) = \Phi_\cB\left(\env{A}\right)$, we may take $\sum_is_\cB\left(a_i\right)t_\cB\left(\op{b_i}\right)$ as generic element in $\ker\left(\gr_0\left(\phi\right)\right)$. Then
\[0 = \gr_0\left(\phi\right)\left(\sum_is_\cB\left(a_i\right)t_\cB\left(\op{b_i}\right)\right) = \sum_is_{\cB'}\left(a_i\right)t_{\cB'}\left(\op{b_i}\right).\]
By applying $\varepsilon_{\cB'}$ to both sides we find out that $\sum_ia_ib_i = 0$ in $A$, whence
\[\sum_ia_i \otimes \op{b_i} = \sum_i\left(a_i \otimes \op{1_A}\right)\left(b_i \otimes \op{1_A} - 1_A \otimes \op{b_i}\right)\]
in $\env{A}$ and so
\[\sum_is_\cB\left(a_i\right)t_\cB\left(\op{b_i}\right) = \sum_is_\cB\left(a_i\right)\left(s_{\cB}\left(b_i\right) - t_\cB\left(\op{b_i}\right)\right) \in \env{A}\cdot \langle s_\cB - t_\cB\rangle.\]
Being $\phi$ injective on $\env{A}\cdot \prim\left(\cB\right)$, we conclude that $\sum_is_\cB\left(a_i\right)t_\cB\left(\op{b_i}\right) = 0$ and hence $\gr_0\left(\phi\right)$ is injective. 

To prove that $\gr_1\left(\phi \right)$ is injective, notice that an element in $F_1\left(\cB\right)/F_0\left(\cB\right)$ is of the form
\begin{align*}
& \sum_is_\cB\left(a_i\right)t_\cB\left(\op{b_i}\right) + \sum_js_\cB\left(a'_j\right)t_\cB\left(\op{{b'_j}}\right)X_js_\cB\left(a''_j\right)t_\cB\left(\op{{b''_j}}\right) + F_0\left(\cB\right) \\
 & \stackrel{\eqref{eq:PRIM}}{=} \sum_js_\cB\left(a'_ja''_j\right)t_\cB\left(\op{\left(b''_jb'_j\right)}\right)X_j + F_0\left(\cB\right)
\end{align*}
for $X_j \in \prim\left(\cB\right)$ and $a_i,b_i,a_j',b_j',a_j'',b_j'' \in A$ for all $i,j$. Therefore, we may assume that $\sum_is_\cB\left(a_i\right)t_\cB\left(\op{b_i}\right)X_i + F_0\left(\cB\right)$ is a generic element belonging to $\ker\left(\gr_1\left(\phi\right)\right)$. This implies that
\[0 = \gr_1\left(\phi\right)\left(\sum_is_\cB\left(a_i\right)t_\cB\left(\op{b_i}\right)X_i + F_0\left(\cB\right)\right) = \sum_is_{\cB'}\left(a_i\right)t_{\cB'}\left(\op{b_i}\right)\phi\left(X_i\right) + F_0\left(\cB'\right)\]
and hence there exists $\sum_js_{\cB'}\left(a_j'\right)t_{\cB'}\left(\op{{b_j'}}\right) \in F_0\left(\cB'\right)$ such that
\[\sum_is_{\cB'}\left(a_i\right)t_{\cB'}\left(\op{b_i}\right)\phi\left(X_i\right) + \sum_js_{\cB'}\left(a_j'\right)t_{\cB'}\left(\op{{b_j'}}\right) = 0\]
in $\cB'$. By applying $\varepsilon_{\cB'}$ again we find out that $\sum_ja_j'b_j' = 0$ and hence 
\[\sum_js_{\cB'}\left(a_j'\right)t_{\cB'}\left(\op{{b_j'}}\right) = \sum_js_{\cB'}\left(a_j'\right)\left(s_{\cB'}\left(b_j'\right) - t_{\cB'}\left(\op{{b_j'}}\right)\right) \in \env{A}\cdot \langle s_{\cB'} - t_{\cB'}\rangle.\]
Summing up,
\begin{align*}
0 & = \sum_is_{\cB'}\left(a_i\right)t_{\cB'}\left(\op{b_i}\right)\phi\left(X_i\right) + \sum_js_{\cB'}\left(a_j'\right)\left(s_{\cB'}\left(b_j'\right) - t_{\cB'}\left(\op{{b_j'}}\right)\right) \\ 
 & = \phi\left(\sum_is_{\cB}\left(a_i\right)t_{\cB}\left(\op{b_i}\right)X_i + \sum_js_{\cB}\left(a_j'\right)\left(s_{\cB}\left(b_j'\right) - t_{\cB}\left(\op{{b_j'}}\right)\right)\right),
\end{align*}
but, being $\phi$ injective on $\env{A}\cdot \prim(\cB)$, this yields that
\[ 0 = \sum_is_{\cB}\left(a_i\right)t_{\cB}\left(\op{b_i}\right)X_i + \sum_js_{\cB}\left(a_j'\right)\left(s_{\cB}\left(b_j'\right) - t_{\cB}\left(\op{{b_j'}}\right)\right)\]
in $\cB$ and hence 
\[\sum_is_\cB\left(a_i\right)t_\cB\left(\op{b_i}\right)X_i + F_0\left(\cB\right) = 0.\]

Finally, let us prove that $\gr_n(\phi )$ is injective for all $n\geq 1$ by induction. We just showed the case $n=1$. Assume that $\gr_1(\phi ),\ldots,\gr_n(\phi )$ are all injective for a certain $n \geq 1$ and consider an element $z \in \ker\left(\gr_{n+1}(\phi )\right)$. Consider also the canonical projections
\[p_{h,k}^\cB  : \bigoplus_{i+j=n+1}\gr_i(\cB ) \tensor{A} \gr_j(\cB ) \to \gr_h(\cB ) \tensor{A} \gr_k(\cB )\]
for $h+k=n+1$, as in \eqref{eq:Deltacomps}. For all $p,q$ such that $p+q = n+1$ we have that
\begin{equation}\label{eq:pphi}
p_{h,k}^{\cB'}  \circ \left({\displaystyle\bigoplus_{i+j=n+1}\gr_i(\phi ) \tensor{A} \gr_j(\phi )}\right) = \left(\gr_h(\phi ) \tensor{A} \gr_k(\phi )\right) \circ p_{h,k}^{\cB}.
\end{equation}
Therefore, for all $1 \leq h \leq n$ we have that
\begin{align*}
0 & = \left(p_{h,k}^{\cB'} \circ \Delta_{\gr({\cB'})}^{[n+1]} \circ\left(\gr_{n+1}(\phi  )\right)\right)\left(z\right) \stackrel{\eqref{eq:Deltahk}}{=} \left(p_{h,k}^{\cB'} \circ \left(\bigoplus_{i+j=n+1}\gr_i(\phi ) \tensor{A} \gr_j(\phi )\right) \circ \Delta^{[n+1]}_{\gr(\cB)}\right)\left(z\right) \\
 & \stackrel{\eqref{eq:pphi}}{=} \left(\left(\gr_h(\phi ) \tensor{A} \gr_k(\phi )\right) \circ p_{h,k}^{\cB} \circ \Delta^{[n+1]}_{\gr(\cB)}\right)\left(z\right).
\end{align*}
By the inductive hypothesis and projectivity of $\gr_s({\cB})$ and $\gr_s({\cB'} )$ as left $\env{A}$-modules for all $s\geq 0$, we know that $\gr_h(\phi ) \tensor{A} \gr_k(\phi )$ is injective and hence
\[
\left(p_{h,k}^{\cB} \circ \Delta^{[n+1]}_{\gr(\cB)}\right)\left(z\right) = 0
\]
for all $h+k=n+1$, $1 \leq h \leq n$. 
Since $\gr\left(\cB\right)$ is strongly graded by hypothesis, $p_{h,k}^{\cB} \circ \Delta^{[n+1]}_{\gr(\cB)}$ is injective and hence $z = 0$.
\end{proof}

\begin{theorem}\label{thm:MM}
Let $\cB $ be an $A$-bialgebroid. Then we have an isomorphism
\[\cB \cong A \odot  U_\K\left(L\right) \odot  A\]
for an $A$-anchored Lie algebra $(L,\omega)$ if and only if 
\begin{enumerate}[label=(CM\arabic*),ref=\textit{(CM\arabic*)}]
\item\label{item:CM2} $L$ is a subalgebra of $\prim(\cB)$ and $\prim(\cB) \cong L \ltimes_\delta \langle s_\cB-t_\cB \rangle$,
\item\label{item:CM1} $\cB $ is graded projective and primitively generated, 
\item\label{item:CM3} the left $\env{A}$-submodule of $\cB $ generated by $L$ is $0$ (in which case we require $\eta_\cB$ to be injective) or it is free and generated by a $\K$-basis of $L$,
\item\label{item:CM4} $\env{A}\cdot \langle s-t\rangle \cap \env{A}\cdot L = 0$ (in particular, $\env{A}\cdot \prim(\cB) = \left(\env{A}\cdot L\right) \oplus \left(\env{A}\cdot \langle s-t\rangle\right)$).
\end{enumerate}
\end{theorem}

\begin{proof}
We want to apply Theorem \ref{thm:HR} to show that the conditions listed are sufficient. 
First of all, let us prove that for any $A$-anchored Lie algebra $(L,\omega)$ the Connes-Moscovici's bialgebroid $\cB_L$ is graded projective and that $\gr\left(\cB_L\right)$ is strongly graded (we already know that $\cB_L$ is primitively generated from Example \ref{ex:CMfiltration} and Remark \ref{rem:primgen}). 

In view of \eqref{eq:CMfiltration}, we know that $F_n\left(\cB _L\right) = A \otimes F_n(U_\K(L)) \otimes A$. By exactness of the tensor product over a field, the short exact sequence of $\K$-vector spaces
\[
\xymatrix{
0 \ar[r] & F_{n-1}(U_\K(L)) \ar[r] & F_n(U_\K(L)) \ar[r] & \gr_n(U_\K(L)) \ar[r] & 0
}
\]
induces a short exact sequence of left $\env{A}$-modules
\[
\xymatrix{
0 \ar[r] & A \otimes F_{n-1}(U_\K(L)) \otimes A \ar[r] & A \otimes F_n(U_\K(L)) \otimes A \ar[r] & A \otimes \gr_n(U_\K(L)) \otimes A \ar[r] & 0.
}
\]
Therefore, we have $\gr_n(\cB _L) \cong A \otimes \gr_n(U_\K(L)) \otimes A$ as left $\env{A}$-modules. In particular, $\gr_n\left(\cB_L\right)$ is a free left $\env{A}$-module.

To show that $\gr\left(\cB_L\right)$ is strongly graded, consider an element
\[z \coloneqq \sum_{\alpha,\beta} a_\alpha \otimes u_{\alpha,\beta} \otimes b_\beta \ \in \ A \otimes \gr_{n}(U) \otimes A\]
such that $ \left(p_{h,k}^{\cB_L} \circ \Delta^{[n]}_{\gr(\cB_L)}\right)\left(z\right) = 0$ for some $h,k$ satisfying $h+k = n$, where the elements $\{a_\alpha\}_\alpha$ in $A$ are linearly independent over $\K$ as well as the elements $\{b_\beta\}_\beta$. Now, consider the commutative diagram
\[
\xymatrix @C=60pt{
A \otimes \gr_{n}(U) \otimes A \ar[r]^-{a_\alpha^* \otimes \gr_{n}(U) \otimes b_\beta^*} \ar[d]_-{\Delta^{[n]}_{\gr(\cB_L)}} & \gr_{n}(U) \ar[dd]^-{\Delta^{[n]}_{\gr(U)}} \\
{\displaystyle\bigoplus_{i+j = n} \left(A \otimes \gr_i(U) \otimes A\right) \tensor{A} \left(A \otimes \gr_j(U) \otimes A\right)} \ar[d]_-{p^{\cB_L}_{h,k}} &  \\
\left(A \otimes \gr_h(U) \otimes A\right) \tensor{A} \left(A \otimes \gr_k(U) \otimes A\right) \ar[d]_-{\cong} & {\displaystyle\bigoplus_{i+j = n} \gr_i(U) \otimes \gr_j(U) } \ar[d]^-{p_{h,k}^{\gr(U)}} \\
A \otimes \gr_h(U) \otimes A \otimes \gr_k(U) \otimes A \ar[r]_-{a_\alpha^* \otimes \gr_h(U) \otimes 1_A^* \otimes \gr_k(U) \otimes b_\beta^*} & \gr_h(U) \otimes \gr_k(U).
}
\]
If we plug $z$ in it and we recall that $\gr(U)$ is strongly graded (that is, $p_{h,k}^{\gr(U)} \circ \Delta^{[n]}_{\gr(U)}$ is injective for all $n \geq 0$ and for all $h+k = n$) then we find $u_{\alpha,\beta} = 0$ for all $\alpha,\beta$ and hence $z = 0$.

Now, the inclusion of $A$-anchored Lie algebras $L \to \prim(\cB)$ extends uniquely of to a morphism of $A$-bialgebroids $\Psi:\cB_L \to \cB$, in view of the universal property of $\cB_L$ (Corollary \ref{cor:unit}). Moreover, similarly to what we did in the proof of Corollary \ref{cor:surjectivity}, one can show that $\Psi$ is surjective (because $\cB$ is primitively generated and $\bP(\Psi)$ is surjective).
Therefore, to conclude by applying Theorem \ref{thm:HR} we are left to check that the candidate isomorphism $\Psi$ is injective when restricted to $\env{A}\cdot \prim(\cB_{L })$. Since $\prim(\cB_{L }) = \left(1_A \otimes L  \otimes 1_A \right) \oplus \left\langle s_{\cB_{L }} - t_{\cB_{L }} \right\rangle$, a generic element in $\env{A}\cdot \prim(\cB_{L })$ is of the form
\[
\sum_{i,j} a_{i,j} \otimes x_i \otimes b_{i,j} + \sum_{h,k}a''_{h,k}a'_h \otimes 1_U \otimes b_{h,k}'' - a_{h,k}'' \otimes 1_U \otimes a'_hb_{h,k}''
\]
for $X_i \in L $ and $a_{i,j},b_{i,j},a_{h,k}'',b_{h,k}'',a_h'\in A$, where we may assume that the $X_k$'s are elements of a $\K$-basis of $L $, without loss of generality. Thus,
\begin{align*}
0 & = \Psi\left(\sum_{i,j} a_{i,j} \otimes x_i \otimes b_{i,j} + \sum_{h,k}a''_{h,k}a'_h \otimes 1_U \otimes b_{h,k}'' - a_{h,k}'' \otimes 1_U \otimes a'_hb_{h,k}''\right) \\
 & = \sum_{i,j} s_\cB\left(a_{i,j}\right)t_\cB\left(b_{i,j}\right)X_i + \sum_{h,k} s_\cB\left(a_{h,k}'' a'_h\right)t_\cB\left(b_{h,k}''\right) - \sum_{h,k}s_\cB\left(a_{h,k}''\right)t_\cB\left(a'_hb_{h,k}''\right).
\end{align*}
By \ref{item:CM4}, this entails that
\begin{gather}
0 = \sum_{i,j} s_\cB\left(a_{i,j}\right)t_\cB\left(b_{i,j}\right)X_i \qquad \text{and} \label{eq:supertech1} \\
0 = \sum_{h,k} s_\cB\left(a_{h,k}'' a'_h\right)t_\cB\left(b_{h,k}''\right) - s_\cB\left(a_{h,k}''\right)t_\cB\left(a'_hb_{h,k}''\right). \label{eq:supertech2}
\end{gather}
By \ref{item:CM3}, relation \eqref{eq:supertech1} yields that 
\[\sum_{j} a_{i,j} \otimes b_{i,j} = 0 \]
in $\env{A}$ for all $i$. Relation \eqref{eq:supertech2}, instead, implies that
\[ 0 = \eta_\cB\left(\sum_{h,k}a''_{h,k}a'_h \otimes b_{h,k}'' - a_{h,k}'' \otimes a'_hb_{h,k}''\right).\]
However, since $\env{A}\cdot L $ is a free left $\env{A}$-module with action given via $\eta_\cB$ (or $\eta_\cB$ is injective by hypothesis), $\eta_\cB$ itself has to be injective and hence
\[0 = \sum_{h,k}a''_{h,k}a'_h \otimes b_{h,k}'' - a_{h,k}'' \otimes a'_hb_{h,k}'',\]
which, in turn, yields
\[0 = \sum_{h,k}a''_{h,k}a'_h \otimes 1_U \otimes b_{h,k}'' - a_{h,k}'' \otimes 1_U \otimes a'_hb_{h,k}''.\]
Summing up, \ref{item:CM2} ensures the existence of a morphism $\Psi$ and \ref{item:CM1} entails that $\Psi$ is surjective. Conditions \ref{item:CM3} and \ref{item:CM4}, instead, allow us to conclude that $\Psi$ is injective on $\env{A}\cdot \prim\left(\cB_L\right)$ and hence, by Theorem \ref{thm:HR}, $\Psi$ is injective on $\cB_L$. Thus, $\Psi$ is an isomorphism. 

The fact that the conditions \ref{item:CM2} -- \ref{item:CM4} are necessary is clear
\end{proof}

\begin{remark}
In the context of the proof above, observe that if $L = 0$, then $\prim(\cB) = \langle s_\cB - t_\cB\rangle$. If moreover $\cB$ is primitively generated, then $\Psi : A \odot  \K \odot  A \to \cB$ is surjective and it coincides with $\eta_\cB : \env{A} \to \cB$ up to the isomorphism $A\odot  \K \odot  A \cong \env{A}$. This is the point where injectivity of $\eta_\cB$ enters the picture.
\end{remark}

\begin{example}
Let $A = \mathsf{Mat}_n(\K)$ for $n \geq 2$ and let $\cB = \End{\K}{A}$ as in Example \ref{ex:bialgebroids}\ref{item:exb}.
It follows from the Skolem-Noether theorem that every derivation of $A$ is inner. 
In particular, $\prim(\cB) = \langle s-t\rangle$ and conditions \ref{item:CM2} and \ref{item:CM4} are satisfied. Furthermore, as we have seen in Example \ref{ex:matrices}, $\cB$ is also primitively generated and, in fact, $\gr_n(\cB) = 0$ for all $n \geq 1$. 

In order to apply Theorem \ref{thm:MM}, we are left to show that $\eta_\cB$ is injective (notice that we already know it is surjective), but a straightforward computation reveals that $\eta_\cB$ coincides with the composition of isomorphisms
\[\mathsf{Mat}_n(\K) \otimes \op{\mathsf{Mat}_n(\K)} \xrightarrow{\mathsf{Mat}_n(\K) \otimes (-)^{\mathsf{T}}} \mathsf{Mat}_n(\K) \otimes \mathsf{Mat}_n(\K) \cong \mathsf{Mat}_{n^2}(\K) \cong \End{\K}{\mathsf{Mat}_n(\K)}.\]
Therefore, \ref{item:CM1} and \ref{item:CM3} are satisfied as well and, by Theorem \ref{thm:MM}, $\cB \cong A \odot U_\K(0) \odot A$.
\end{example}


\subsection{Bialgebroids over commutative algebras}\label{ssec:commutative}

A slightly more favourable situation is provided by the case of bialgebroids over a commutative base.

Let us assume henceforth that $A$ is a commutative $\K$-algebra. This implies that now we can consider the target $t$ of an $A$-bialgebroid $\cB$ as an algebra map $t : A \to \cB$ and hence we will omit the $\op{(-)}$. 
By Lemma \ref{lem:LieIdeal}, we may consider the quotient Lie algebra
\[ \cl{\prim(\cB)} \coloneqq \prim(\cB)/\langle s-t \rangle,\]
which is $A$-anchored with anchor $\cl{\omega}_\cB$ induced by $\omega_\cB$, because now $\langle s-t \rangle$ is abelian with zero anchor. This induces a well-defined functor
\[\bP':\Bialgd_{A} \to \ALie{A}, \qquad \cB \mapsto \cl{\prim(\cB)}.\]

As we have seen in Proposition \ref{prop:primitive} and in Corollary \ref{cor:primitive}, $\cB_L$ satisfies the additional property that $\prim(\cB_L) \cong \cl{\prim(\cB_L)} \ltimes_{\delta} \langle s-t \rangle$ as $A$-anchored Lie algebras.
If we restrict our attention to the full subcategory $\cl{\Bialgd}_A$ of $\Bialgd_A$ composed by all those $A$-bialgebroids $\cB$ such that $\prim(\cB) \cong \cl{\prim(\cB)} \ltimes_{\delta} \langle s-t\rangle$ as $A$-anchored Lie algebras, then the functors $\bB$ and $\bP'$ induce functors
\[\cl{\bP} : \cl{\Bialgd}_{A} \to \ALie{A}, \qquad \cB \mapsto \cl{\prim(\cB)}\]
and
\[\cl{\bB}:\ALie{A} \to \cl{\Bialgd}_{A}, \qquad (L, \omega) \mapsto A \odot  U_\K(L) \odot  A.\]
It can be shown that, in this case, we always have a natural isomorphism
\[\cl{\gamma}_L \coloneqq \left(L\xrightarrow{\gamma_L} \prim(\cB_L) \twoheadrightarrow \cl{\prim(\cB_L)}\right)\]
inducing a surjective map
\[
\Bialgd_A\left(\cB_L,\cB\right) \to \ALie{A}\left((L,\omega),\left(\cl{\prim(\cB)},\cl{\omega}_\cB\right)\right),
\]
which, however, is not injective in general (that is, $\cl{\bB}$ and $\cl{\bP}$ are not adjoint functors). 
Nevertheless, we may always consider a ``preferred'' morphism of $A$-bialgebroids
\[
\cl{\epsilon}_\cB \coloneqq \Big(A \odot  U_\K\big(\cl{\prim(\cB)}\big) \odot  A \xrightarrow{A \odot  U_\K(\iota_\cB) \odot  A} A \odot  U_\K\big(\prim(\cB)\big) \odot  A \xrightarrow{\epsilon_\cB} \cB\Big)
\]
induced by a chosen injection $\iota_\cB: \cl{\prim(\cB)} \to \prim(\cB)$ and by the counit of the adjunction in Theorem \ref{th:main}. Since the hypothesis on $\prim(\cB)$ ensures that
\[\bP(\cl{\epsilon}_\cB):\prim\left(A \odot  U_\K\big(\cl{\prim(\cB)}\big) \odot  A\right) \to \prim(\cB)\] 
is an epimorphism, we may conclude that $\cB$ is primitively generated if and only if $\cl{\epsilon}_\cB$ is surjective, as in Corollary \ref{cor:surjectivity}, and we may restate Theorem \ref{thm:MM} in the present framework.

\begin{theorem}\label{thm:MMbis}
Let $\cB $ be an $A$-bialgebroid over a commutative algebra $A$. Then
\[\cB \cong A \odot  U_\K\left(\cl{\prim(\cB)}\right) \odot  A\]
if and only if 
\begin{enumerate}[label=(CM\arabic*),ref=\textit{(CM\arabic*)}]
\item\label{item:CM2bis} $\prim(\cB) \cong \cl{\prim(\cB)} \ltimes_\delta \langle s_\cB-t_\cB \rangle$,
\item\label{item:CM1bis} $\cB $ is graded projective and primitively generated, 
\item\label{item:CM3bis} the left $\env{A}$-submodule of $\cB $ generated by $\cl{\prim(\cB )}$ is $0$ (in which case we require $\eta_\cB$ to be injective) or it is free and generated by a $\K$-basis of $\cl{\prim(\cB )}$,
\item\label{item:CM4bis} $\env{A}\cdot \langle s-t\rangle \cap \env{A}\cdot \cl{\prim(\cB)} = 0$.
\end{enumerate}
\end{theorem}

\begin{example}
Let $A \coloneqq \C[X]$ and consider the $A$-bialgebroid
\[
\cH \coloneqq \C\left[x,y,t,z,\frac{1}{t}\right]
\]
studied in \cite[\S5.6]{LaiachiPepe-big} and inspired by the coordinate ring of the Malgrange's groupoid. Its bialgebroid structure is uniquely determined by
\begin{gather*}
s_\cH(X) = x, \qquad t_\cH(X) = y, \qquad \varepsilon_\cH(x) = X = \varepsilon_\cH(y), \qquad \varepsilon_\cH(t) = 1, \qquad \varepsilon_\cH(z) = 0, \\
\Delta_\cH(x) = x \tensor{A} 1, \qquad \Delta_\cH(y) = 1 \tensor{A} y, \qquad \Delta_\cH(t) = t \tensor{A} t, \qquad \Delta_\cH(z) = z \tensor{A} t + t^2 \tensor{A} z
\end{gather*}
and ordinary multiplication and unit. The ideal $\cI = \langle t-1\rangle$ generated by $t-1$ in $\cH$ is a bi-ideal (it is an ideal by construction and a coideal by \cite[\S2.4]{BrzezinskiWisbauer}) and hence the quotient $\cH/\cI$ is an $A$-bialgebroid. It can be identified with $\cB\coloneqq \C[u,v,w]$ with
\begin{gather*}
s(X) = u, \qquad t(X) = w, \qquad \varepsilon(u) = X = \varepsilon(w), \qquad \varepsilon(v) = 0, \\
\Delta(u) = u \tensor{A} 1, \qquad \Delta(w) = 1 \tensor{A} w, \qquad \Delta(v) = v \tensor{A} 1 + 1 \tensor{A} v.
\end{gather*}
The space of primitive elements of $\cB$ is
\[\prim(\cB) = \langle u^i - w^i \mid i \geq 0 \rangle \oplus \C v\]
with $\cl{\prim(\cB)} \cong \C v$, whence \ref{item:CM2bis}, \ref{item:CM3bis} and \ref{item:CM4bis} are satisfied (in this case, $\delta \equiv 0$ by Lemma \ref{lem:sdprod}, as everything is commutative). Concerning \ref{item:CM1bis}, we observe that $\cB$ with the foregoing structures is the free left $\env{A}$-module generated by $\{v^k \mid k \geq 0\}$, whence it is primitively generated and graded projective (even free). Thus, Theorem \ref{thm:MMbis} ensures that
\[\cB \cong A \odot U_\C(\C v) \odot A.\]
The same conclusion could have been drawn by observing that $\cB \cong \C[X] \otimes \C[Y] \otimes \C[X]$ is the scalar extension commutative (Hopf) $\C[X]$-bialgebroid obtained from the Hopf algebra $\C[Y]$ and that, as Hopf algebras, $\C[Y] \cong U_\C(\C Y)$.
\end{example}


\subsection{Final Remarks}\label{ssec:finalremarks}

An additional step which deserve to be taken is to restrict the attention further to those $A$-bialgebroids $\cB$ over a commutative algebra $A$ such that $s_\cB=t_\cB$ (for example, cocommutative $A$-bialgebroids). However, in this case the Connes-Moscovici's construction is not the correct construction to look at. One may prove that the assignment
\[A \odot  U_\K(L) \odot  A \to A ~\#~  U_\K(L) , \qquad a \otimes u \otimes b \mapsto ab \otimes u\]
provides a surjective homomorphism of $A$-bialgebroids with kernel the ideal generated by $\langle s_{\cB_L}-t_{\cB_L} \rangle$ in $\cB_L$. The $A$-bialgebroid structure on $A ~\#~  U_\K(L)$ is that of extension of scalars with trivial coaction on $A$, that is to say,
\[s = t : a \mapsto a \otimes 1_{U_\K(L)}, \qquad \varepsilon:a \otimes u \mapsto  a\varepsilon(u), \qquad \Delta: a \otimes u \mapsto \sum(a \otimes u_{(1)}) \tensor{A} (1_A \otimes u_{(2)})\]
and semi-direct product algebra structure, that is,
\[(a \otimes u)(b \otimes v) = \sum a(u_{(1)}\cdot b) \otimes u_{(2)}v \qquad \text{and} \qquad 1_{A \# U_\K(L)} = 1_A \otimes 1_{U_\K(L)}.\]
In view of the results from \S\ref{ssec:UEA} and \S\ref{ssec:adj}, the foregoing observations suggest that $A ~\#~  U_\K(L)$ would be the right $A$-bialgebroid construction to consider, in order to recover the universal property of Theorem \ref{thm:UEA} and an adjunction as in Theorem \ref{th:main}. 
Nevertheless, we keep this question for a future investigation.


\end{document}